\documentclass[reqno,twoside]{amsart}
\usepackage{amscd}
\usepackage{amsfonts}
\usepackage{amsmath}
\usepackage{amssymb}
\usepackage{amsthm}
\usepackage{amsaddr}
\usepackage{fancyhdr}
\usepackage{latexsym}
\usepackage[colorlinks=true, pdfstartview=FitV, linkcolor=blue, citecolor=blue, urlcolor=blue]{hyperref}
\usepackage{enumitem}      
\usepackage{mathtools}            
\usepackage{indentfirst} 
\usepackage{color}
\usepackage{caption}          
\usepackage[normalem]{ulem}
\usepackage{upgreek}
\usepackage{dutchcal}
\usepackage{mathrsfs}
\newcommand{\eee}[1]{\begin{equation}#1\end{equation}}

\newcommand{\aaa}[1]{\begin{alignat}{2}#1\end{alignat}}
\newcommand{\ddd}[1]{\begin{equation}\begin{aligned}#1\end{aligned}\end{equation}}
\newcommand{\nn}{\nonumber}
\newcommand{\p}{\partial}

\newcommand{\no}[1]{\left\| #1 \right\|}

\usepackage{accents}

\usepackage{scalerel,stackengine}
\stackMath
\newcommand\wwhat[1]{%
\savestack{\tmpbox}{\stretchto{%
  \scaleto{%
    \scalerel*[\widthof{\ensuremath{#1}}]{\kern-.6pt\bigwedge\kern-.6pt}%
    {\rule[-\textheight/2]{1ex}{\textheight}}
  }{\textheight}%
}{0.5ex}}%
\stackon[1pt]{#1}{\tmpbox}%
}
\usepackage{thmtools}
\declaretheoremstyle[headfont=\normalfont\bfseries, bodyfont=\itshape, spaceabove=7pt,spacebelow=7pt]{theorem} 
\theoremstyle{theorem} 
\newtheorem{theorem}{Theorem}[] 
\newtheorem{remark}{Remark}[section]

\newtheorem{lemma}{Lemma}[section]

\theoremstyle{definition}

\allowdisplaybreaks
\numberwithin{equation}{section}
\numberwithin{figure}{section}
\usepackage{geometry}
\usepackage[breakable, theorems, skins]{tcolorbox}
\tcbset{enhanced}

\usepackage{tikz}
\usetikzlibrary{arrows.meta}
\usetikzlibrary{decorations.markings}
\tikzset{->-/.style={decoration={
  markings,
  mark=at position #1 with {\arrow{>}}},postaction={decorate}}}
  \tikzset{middlearrow/.style={
        decoration={markings,
            mark= at position 0.55 with {\arrow{#1}} ,
        },
        postaction={decorate}
    }
}
\makeatletter
  \renewcommand\subsection{\@startsection{subsection}{2}%
  \z@{-1\linespacing\@plus-0.7\linespacing}{0.7\linespacing}%
  {\bfseries}}
\makeatother
\let\OLDthebibliography\thebibliography
\renewcommand\thebibliography[1]{
  \OLDthebibliography{#1}
  \setlength{\parskip}{0pt}
  \setlength{\itemsep}{0pt plus 0.3ex}
}
\AtBeginDocument{
   \def\MR#1{}
}

\makeatletter
\@namedef{subjclassname@2020}{%
  \textup{2020} Mathematics Subject Classification}
\makeatother

\begin{document}

\title{A new approach for the analysis of evolution partial differential equations on a finite interval}

\author{T\"urker \"Ozsar{\i}$^a$, Dionyssios Mantzavinos$^b$, Konstantinos Kalimeris$^c$}
	
\address{
\normalfont $^a$Department of Mathematics, Bilkent University, 06800 Ankara, Turkey
\\
\normalfont $^b$Department of Mathematics, University of Kansas, Lawrence, KS 66045, USA
\\
\normalfont $^c$Mathematics Research Center, Academy of Athens, 115 27 Athens, Greece
} 
\email{turker.ozsari@bilkent.edu.tr, mantzavinos@ku.edu \textnormal{(corresponding author)}, kkalimeris@academyofathens.gr}

\thanks{\textit{Acknowledgements.} DM gratefully acknowledges support from the U.S. National Science Foundation (NSF-DMS 2206270 and NSF-DMS 2509146) and the Simons Foundation (SFI-MPS-TSM-00013970). Furthermore, DM is thankful to the Department of Mathematics of Bilkent University, Ankara, Turkey, for their warm hospitality during March of 2025, when part of this work was undertaken. 
TÖ's research is supported by BAGEP 2020 Young Scientist Award.
KK acknowledges support by the Sectoral Development Program (SDP 5223471) of the Greek Ministry of Education, Religious Affairs and Sports, through the National Development Program (NDP) 2021-25, grant no 200/1029. 
The authors are grateful to the reviewers of the manuscript for their constructive remarks that led to its improvement. 
}
\date{November 4, 2025. \textit{Revised}: May 13, 2026}

\begin{abstract}
We show that, for certain evolution partial differential equations, the solution on a finite interval $(0,\ell)$ can be reconstructed as a superposition of restrictions to $(0,\ell)$ of solutions to two associated partial differential equations posed on the half-lines $(0,\infty)$ and $(-\infty,\ell)$. Determining the appropriate data for these half-line problems amounts to solving an inverse problem, which we formulate via the unified transform of Fokas (also known as the Fokas method) and address via a fixed point argument in $L^2$-based Sobolev spaces, including fractional ones through interpolation techniques. We illustrate our approach through two canonical examples, the heat equation and the Korteweg-de Vries (KdV) equation, and provide numerical simulations for the former example. We further demonstrate that the new approach extends to more general evolution partial differential equations, including those with time-dependent coefficients. A key outcome of this work is that spatial and temporal regularity estimates for problems on a finite interval can be directly derived from the corresponding estimates on the half-line. These results can, in turn, be used to establish local well-posedness for related nonlinear problems, as the essential ingredients are the linear estimates within nonlinear frameworks.
\end{abstract}

\maketitle
\markboth
{T\"urker \"Ozsar{\i}, Dionyssios Mantzavinos, Konstantinos Kalimeris}
{A new approach for the analysis of evolution partial differential equations on a finite interval}


\section{Introduction}

\textit{Initial-boundary value problems} for evolution partial differential equations arise naturally in a wide range of applications associated, in particular, with various areas of physics and engineering, including water waves and control theory. Such problems model phenomena taking place over a spatial domain that involves a boundary, like the half-line $(0, \infty)$ or the finite interval $(0, \ell)$ in one spatial dimension. In contrast to the more standard \textit{initial value (Cauchy) problems}, which are formulated either on the entire space  or on boundaryless manifolds and only require the prescription of initial conditions, initial-boundary value problems must also be supplemented with appropriate \textit{boundary conditions}. 

Furthermore, while in some cases zero (homogeneous) boundary conditions can be used to model certain phenomena, \textit{nonzero} (nonhomogeneous) boundary conditions are the ones present in the vast majority of applications. Hence, the study of \textit{nonhomogeneous} initial-boundary value problems is of critical importance. At the same time, this task can become quite intricate, even at the linear level. In particular, it is worth noting that, while the Cauchy problem for any linear evolution equation can be easily solved by taking a Fourier transform with respect to the spatial variable, such a classical spatial transform is not available for nonhomogeneous initial-boundary value problems that involve linear evolution equations of spatial order higher than two \cite{f2008}. This fact directly affects the analysis of \textit{nonlinear} equations.

In the case of nonlinear dispersive models like the Korteweg-de Vries (KdV) and the nonlinear Schr\"odinger (NLS) equation, the proof of Hadamard well-posedness (existence, uniqueness, and continuous dependence of the solution on the data) for the Cauchy problem crucially relies on a Picard iteration scheme and linear estimates established through the Fourier transform solution of the linearized problem. Hence, the absence of the Fourier transform from the initial-boundary value problem setting  poses an immediate challenge right at the beginning of the analysis.
As a result, new techniques had to be developed for the study of nonlinear dispersive (and, more generally, evolution) equations specifically in domains with a boundary. 

The three main methods available in the literature for proving the well-posedness of initial-boundary value problems are:
\begin{enumerate}[label=(\arabic*), leftmargin=6mm, topsep=1mm, itemsep=1mm]
\item The \textit{temporal Laplace transform method} of Bona, Sun and Zhang, which was introduced in \cite{bsz2002} for the KdV equation on the half-line and has since been used in several other works, e.g. \cite{bsz2006,bsz2008,kai2013,ozs2015,bo2016,et2016}; 
\item The \textit{boundary forcing operator method} of Colliander and Kenig, developed for the generalized KdV equation on the half-line \cite{ck2002} and later employed by Holmer for the KdV and NLS equations on the half-line \cite{h2005,h2006} (and, more recently, in  \cite{c2017,cc2020});
\item The \textit{unified transform method} introduced in~\cite{fhm2017,fhm2016} for the NLS and KdV equations on the half-line, which takes advantage of the unified transform (also known as the Fokas method) \cite{f1997,f2008} as the analogue of the Fourier transform in domains with a boundary and has been consistently developed through several works in recent years, e.g. see~\cite{oy2019,hm2020,hm2022,ko2022,hy2022-jde,mo2025}. 
\end{enumerate}

\vskip 2mm

\textit{The purpose of the present work is to demonstrate that the proof of well-posedness of a wide class of initial-boundary value problems on the finite interval $(0, \ell)$ can be reduced to that of suitable initial-boundary value problems on the positive and negative half-lines $(0, \infty)$ and $(-\infty, \ell)$.}  

\vskip 3mm

The new approach leading to this reduction relies on the powerful linear solution formulae that form the core of the unified transform method outlined above and, more specifically, on the exponentially decaying integrands that involve the boundary data in these formulae. 

As the difference in the analysis of the half-line and finite interval problems lies only in the derivation of the \textit{linear estimates} needed for the contraction mapping argument of the Picard iteration, the ``finite interval to half-line'' reduction introduced through the present work is \textit{only} required at the level of the \textit{forced linear counterpart} of the initial-boundary value problem under consideration. 
We introduce our approach via two fundamental examples: (i)~the forced heat equation and (ii)~the forced linearized KdV equation, both formulated on the finite interval $(0, \ell)$ with nonzero Dirichlet and Neumann boundary data, as appropriate. 

We emphasize that nonlinear analogues of each of these two linear models have already been studied \textit{directly} on the finite interval. For example, the local Hadamard well-posedness of a reaction-diffusion equation whose linear part corresponds to the heat equation was established in \cite{hmy2019-rd}. Moreover, local and global well-posedness results for the KdV equation on a finite interval were obtained in \cite{bsz2003,h2006,fam2007,hmy2019-kdv} via the three different methods outlined above, including the treatment of lower regularity assumptions via Bourgain spaces. In particular, \cite{h2006} used operators associated with half-line
problems, whereas \cite{fam2007} constructed the solution of the linear interval problem as a sum of the solution to a negative half-line problem and another linear interval problem.

In the light of the new approach introduced here, the results of these works, along with  other works in the literature on the direct analysis of nonlinear evolution equations on a finite interval such as~\cite{lzz2020,mmo2026},  can instead be deduced from their half-line counterparts (e.g. see \cite{bsz2002,ck2002,fam2004,h2006,fhm2016,lzz2017,f2024,amo2024}) upon solving a certain unidimensional integral equation originating from the exact solution formulae of the associated linear half-line problems (e.g. see \eqref{a-int-eq} and \eqref{a-eq-kdv}). 
For nonlinear equations whose linear part corresponds to the heat equation, the essence of our approach lies in the following key result:

\begin{theorem}[Finite interval to half-line I]\label{heat-t}
Let $m \geq 0$, $\ell>0$ and  $0 < T \leq \dfrac{\sqrt \pi e^{\frac 32}}{2 \cdot 3^{\frac 54}} \, \ell^2$. 
Then, for Dirichlet boundary data $g \in H^m(0, T)$ such that $g^{(n)}(0) = 0$ for all integers $0\leq n < m-\frac 12$, there exist functions $a, b \in H^m(0, T)$ such that the solution $q(x, t)$ to the heat equation finite interval problem
\begin{equation}\label{q-fi-ibvp-i}
\begin{aligned}
&q_t - q_{xx} = 0, \quad x \in (0, \ell), \ t \in (0, T),
\\
&q(x, 0) = 0, \quad x \in (0, \ell),
\\
&q(0, t) = g(t), \quad q(\ell, t) = 0, \quad t \in (0, T),
\end{aligned}
\end{equation}
can be expressed as the sum
\begin{equation}\label{fi-hl-dec-i}
q(x, t) = v(x, t) \big|_{x \in (0, \ell)} + w(x, t) \big|_{x \in (0, \ell)}
\end{equation}
of the restrictions on $(0, \ell)$ of the solutions to the heat equation half-line problems
\begin{equation}\label{vw-hl-ibvp-i}
\begin{aligned}
&v_t - v_{xx} = 0, \quad x \in (0, \infty), \ t \in (0, T),
\\
&v(x, 0) = 0, \quad x \in (0, \infty),
\\
&v(0, t) = a(t), \quad t \in (0, T),
\end{aligned}
\hspace*{1cm}
\begin{aligned}
&w_t - w_{xx} = 0, \quad x \in (-\infty, \ell), \ t \in (0, T),
\\
&w(x, 0) = 0, \quad x \in (-\infty, \ell),
\\
&w(\ell, t) = b(t), \quad t \in (0, T).
\end{aligned}
\end{equation}
\end{theorem}

Theorem \ref{heat-t} is proved in Section \ref{red-s}. The corresponding result in the case of nonlinear equations whose linear part is given by the linearized KdV equation is established in Section \ref{kdv-s} and reads as follows:
\begin{theorem}[Finite interval to half-line II]\label{kdv-t}
Let $m \geq 0$, $\ell>0$ and  $T=T(\ell)>0$ satisfy the inequality \eqref{T-L2-kdv}. Then, for Dirichlet boundary data $g \in H^m(0, T)$ such that $g^{(n)}(0) = 0$ for all integers $0\leq n < m-\frac 12$, there exist functions $a, b \in H^m(0, T)$ and $c\in H^{m-\frac 13}(0, T)$ such that the solution $q(x, t)$ to the linearized KdV equation finite interval problem
\begin{equation}\label{fi-kdv-red-i}
\begin{aligned}
&q_t + q_x + q_{xxx} = 0, \quad x \in (0, \ell), \ t \in (0, T),
\\
&q(x, 0) = 0, \quad x \in (0, \ell),
\\
&q(0, t) = g(t), \quad q(\ell, t) = 0, \quad q_x(\ell, t) = 0, \quad t \in (0, T),
\end{aligned}
\end{equation}
can be expressed as the sum
\begin{equation}\label{sup-kdv-i}
q(x, t) = v(x, t) \big|_{x \in (0, \ell)} + w(x, t) \big|_{x \in (0, \ell)}
\end{equation}
of the restrictions on $(0, \ell)$ of the solutions to the linearized KdV equation half-line problems
\begin{equation}\label{vw-hllr-i}
\begin{aligned}
&v_t + v_x + v_{xxx} = 0, \quad x \in (0, \infty), \ t \in (0, T),
\\
&v(x, 0) = 0, \quad x \in (0, \infty),
\\
&v(0, t) = a(t), \quad t \in (0, T),
\end{aligned}
\hspace*{1cm}
\begin{aligned}
&w_t + w_x + w_{xxx} = 0, \quad x \in (-\infty, \ell), \ t \in (0, T),
\\
&w(x, 0) = 0, \quad x \in (-\infty, \ell),
\\
&w(\ell, t) = b(t), \quad w_x(\ell, t) = c(t), \quad t \in (0, T).
\end{aligned}
\end{equation}
\end{theorem}

The impact of Theorems \ref{heat-t} and \ref{kdv-t} can be appreciated via the straightforward application of their results on the decompositions detailed in the beginning of Sections \ref{red-s} and \ref{kdv-s}. Indeed, in view of the half-line results of~\cite{hmy2019-rd,bsz2002}, \textit{Theorems \ref{heat-t} and \ref{kdv-t} directly imply the well-posedness of the nonlinear reaction-diffusion and KdV equations studied on the finite interval $(0, \ell)$ in \cite{hmy2019-rd,bsz2003}} --- see Theorems~\ref{lwp-rd-t} and \ref{lwp-kdv-t} respectively --- without the need for studying those finite interval problems on their own right.
Finally, as one may expect,  the constraints on $T$ in Theorems \ref{heat-t} and \ref{kdv-t} can be removed via the iteration argument outlined in Section \ref{ug-ss}, thus extending the validity of our results to arbitrary $T>0$.
\\[2mm]
\textbf{Structure.}
In Section \ref{red-s}, we establish Theorem \ref{heat-t} by exploiting the dissipative nature of the heat equation. In Section \ref{kdv-s}, we prove Theorem \ref{kdv-t} by adapting our approach to the dispersive nature of the linearized KdV equation. Finally, in Section \ref{c-s}, we remark on the uniqueness of solution and global solvability of the integral equation \eqref{a-int-eq} (which provides the basis for the analysis of Section \ref{red-s}), we numerically illustrate the result of Theorem \ref{heat-t}, and we outline the extension of the method of Section \ref{red-s} to the framework of the heat equation with a time-dependent diffusion coefficient.

\section{The Heat Equation}
\label{red-s}

Consider the forced heat equation on a finite interval with nonzero Dirichlet boundary conditions, namely
\begin{equation}\label{u-fi-ibvp}
\begin{aligned}
&u_t - u_{xx} = f(x, t), \quad x \in (0, \ell), \ t \in (0, T),
\\
&u(x, 0) = u_0(x), \quad x \in (0, \ell),
\\
&u(0, t) = g_0(t), \quad u(\ell, t) = h_0(t), \quad t \in (0, T),
\end{aligned}
\end{equation}
where $u=u(x, t)$ and the precise function spaces for the initial datum $u_0$, the boundary data $g_0, h_0$ and the forcing $f$ will be discussed in due course. 

Furthermore, letting $U_0$ and $F$ respectively denote suitable extensions of $u_0$ and $f$ from the finite interval $(0, \ell)$ to the negative half-line $(-\infty, \ell)$, consider the negative half-line problem
\begin{equation}\label{U-hl-ibvp}
\begin{aligned}
&U_t - U_{xx} = F(x, t), \quad x \in (-\infty, \ell), \ t \in (0, T),
\\
&U(x, 0) = U_0(x), \quad x \in (-\infty, \ell),
\\
&U(\ell, t) = h_0(t), \quad t \in (0, T).
\end{aligned}
\end{equation}
The observation that the function  $\breve U(x, t) := U(\ell-x, t)$ satisfies the positive half-line problem
\begin{equation}\label{Utilde-hl-ibvp}
\begin{aligned}
&\breve U_t - \breve U_{xx} = \breve F(x, t) := F(\ell-x, t), \quad x \in (0, \infty), \ t \in (0, T),
\\
&\breve U(x, 0) = \breve U_0(x) := U_0(\ell-x), \quad x \in (0, \infty),
\\
&\breve U(0, t) = h_0(t), \quad t \in (0, T),
\end{aligned}
\end{equation}
which has been estimated in \cite{hmy2019-rd}, allows us to readily extract estimates for the negative half-line problem~\eqref{U-hl-ibvp}. 

Therefore, in order to estimate the finite interval problem~\eqref{u-fi-ibvp}, it suffices to estimate the following \textit{reduced}  interval problem satisfied by the difference $q := u-U|_{x\in(0, \ell)}$:
\begin{equation}\label{q-fi-ibvp}
\begin{aligned}
&q_t - q_{xx} = 0, \quad x \in (0, \ell), \ t \in (0, T),
\\
&q(x, 0) = 0, \quad x \in (0, \ell),
\\
&q(0, t) = g(t) := g_0(t) - U(0, t), \quad q(\ell, t) = 0, \quad t \in (0, T),
\end{aligned}
\end{equation}
which is precisely problem \eqref{q-fi-ibvp-i} with the above choice of $g$. 

The solution of the reduced interval problem \eqref{q-fi-ibvp} can be formally expressed as the sum \eqref{fi-hl-dec-i} of the restrictions on $(0, \ell)$ of the solutions $v(x, t)$ and $w(x, t)$ to the two half-line problems in \eqref{vw-hl-ibvp-i} provided that the relevant boundary data $a(t)$ and $b(t)$ satisfy the conditions
\begin{equation}\label{ab-cond}
\begin{aligned}
a(t) &= g(t) - w(0, t),
\\
b(t) &= -v(\ell, t).
\end{aligned}
\end{equation}
Importantly, the function $\breve w(x, t) := w(\ell-x, t)$ satisfies the positive half-line problem
\begin{equation}\label{w-hl-ibvp}
\begin{aligned}
&\breve w_t - \breve w_{xx} = 0, \quad x \in (0, \infty), \ t \in (0, T),
\\
&\breve w(x, 0) = 0, \quad x \in (0, \infty),
\\
&\breve w(0, t) = b(t), \quad t \in (0, T),
\end{aligned}
\end{equation}
so the two half-line problems in \eqref{vw-hl-ibvp-i} are of the same type. 

\begin{remark}
An alternative but equivalent way to the above decomposition is to subtract from the original interval problem \eqref{u-fi-ibvp} the problem on the positive half-line $(0, \infty)$ with initial datum $U_0$, forcing $F$ and Dirichlet datum $g_0$ at $x=0$. The conditions resulting from this reduction are entirely analogous to \eqref{ab-cond}.
\end{remark}

In view of the decomposition \eqref{fi-hl-dec-i}, the analysis of the reduced interval problem \eqref{q-fi-ibvp} --- and, in turn (as explained earlier), of the full interval problem \eqref{u-fi-ibvp} --- can be extracted from the analysis of the positive half-line problem in~\eqref{vw-hl-ibvp-i}. In other words, 
\textit{the study of the heat equation and its nonlinear counterparts on the finite interval can be carried through the corresponding analysis on the half-line.}
However, in order for the formal decomposition \eqref{fi-hl-dec-i} to be made rigorous, one must prove that there indeed exist functions $a, b$ satisfying the required conditions \eqref{ab-cond}. 

In this regard, we note that in \cite{hmy2019-rd} it was shown that for $\breve U_0 \in H^s(0, \infty)$, $h_0 \in H^{\frac{2s+1}{4}}(0, T)$,  $\breve F \in C_t([0, T]; H_x^s(0, \infty))$ with $\frac 12 < s <\frac 32$ the solution $\breve U$ to the half-line problem \eqref{Utilde-hl-ibvp} belongs to the space $C_t([0, T]; H_x^s(0, \infty)) \cap C_x([0, \infty); H_t^{\frac{2s+1}{4}}(0, T))$. 
In fact, the validity of that result can easily be extended to $0 \leq s < \frac 32$ (see the proof of estimate (2.80) in \cite{moy2025}). Hence, for the interval problem~\eqref{u-fi-ibvp} considered in the present work, we assume that $s\geq 0$ and take $u_0 \in H^s(0, \ell)$, $g_0, h_0 \in H^{\frac{2s+1}{4}}(0, T)$ and $f \in C_t([0, T]; H_x^s(0, \ell))$. In addition, by the Sobolev embedding theorem, for $s>\frac 12$ the data must also satisfy the compatibility conditions $u_0(0) = g_0(0)$ and $u_0(\ell) = h_0(0)$. Hence, since for $s>\frac 12$ Sobolev functions in $H^s$ are continuous, we deduce the following condition for the nonzero boundary datum of problem~\eqref{q-fi-ibvp}:
\eee{\label{g=0}
g(0) \equiv g_0(0) - U(0, 0) = u_0(0) - U_0(0) = 0, \quad s>\frac 12.
}
Moreover, according to the result of \cite{hmy2019-rd} stated above, for any $a \in H^{\frac{2s+1}{4}}(0, T)$ we have $v \in C_t([0, T]; H_x^s(0, \infty)) \cap C_x([0, \infty); H_t^{\frac{2s+1}{4}}(0, T))$ and so any $b$ satisfying \eqref{ab-cond} must belong to $H^{\frac{2s+1}{4}}(0, T)$. In turn, once again via \cite{hmy2019-rd}, this implies that $\breve w \in C_t([0, T]; H_x^s(0, \infty)) \cap C_x([0, \infty); H_t^{\frac{2s+1}{4}}(0, T))$ or, equivalently, $w \in C_t([0, T]; H_x^s(-\infty, \ell)) \cap C_x((-\infty, \ell]; H_t^{\frac{2s+1}{4}}(0, T))$. Therefore, by the existence of traces for $v, w$ when $s>\frac 12$, we find that $a, b$ must satisfy the conditions
\eee{\label{a=0}
\begin{aligned}
&a(0) \equiv v(0, 0) = v(x, 0)\big|_{x=0} = 0, 
\\
&b(0) \equiv w(\ell, 0) = w(x, 0)\big|_{x=\ell} = 0, 
\end{aligned}
\quad s>\frac 12.
}
Note that these conditions are consistent with \eqref{ab-cond} in view of \eqref{g=0}. 

More generally, for any $n\in\mathbb N$ and $s>2n+\frac 12$, which amounts to $\frac{2s+1}{4} > n + \frac 12$, the regularity of the initial and boundary data implies the existence of traces such that, by means of \eqref{q-fi-ibvp}, 
\eee{\label{gn=0}
g^{(n)}(0) \equiv \p_t^n q(0, 0) 
= \p_t^n q(x, t)\big|_{x=t=0}
= \p_x^{2n} q(x, t)\big|_{x=t=0}
= \p_x^{2n} q(x, 0)\big|_{x=0}
= 0.
}
Similarly, for the positive half-line problem in \eqref{vw-hl-ibvp-i}, we have
\eee{\label{an=0}
a^{(n)}(0) \equiv \p_t^n v(0, t)\big|_{t=0} =   \p_t^n v(x, t)\big|_{x=t=0}  = 
 \p_x^{2n} v(x, t)\big|_{x=t=0} =  \p_x^{2n} v(x, 0)\big|_{x=0}  = 0, \quad s>2n + \frac 12,
}
which is consistent with \eqref{ab-cond} in view of \eqref{gn=0}. An analogous equality consistent with \eqref{ab-cond} also holds for $b^{(n)}(0)$.

Theorem \ref{heat-t} on the existence of $a, b$ that satisfy \eqref{ab-cond} is established below. Before giving the proof, we highlight the significance of Theorem \ref{heat-t} via the following nonlinear well-posedness result, which corresponds to Theorem~1.3 in~\cite{hmy2019-rd} and, in our case, follows via a straightforward application of Theorem \ref{heat-t} along with the results of~\cite{hmy2019-rd} on the linear \textit{half-line} problem that were summarized above \eqref{g=0}:
\begin{theorem}[Nonlinear reaction-diffusion on a finite interval]\label{lwp-rd-t}
For $\frac 12 < s < \frac 32$ and $p \in 2\mathbb N + 1$, the finite interval problem for the nonlinear reaction-diffusion equation
\ddd{\label{rd-ibvp}
&u_t - u_{xx} = |u|^{p-1} u, \quad x\in (0, \ell), \ t\in (0, T), 
\\
&u(x, 0) = u_0(x) \in H^s(0, \ell),
\\
&u(0, t) = g_0(t) \in H^{\frac{2s+1}{4}}(0, T), 
\quad
u(\ell, t) = h_0(t) \in H^{\frac{2s+1}{4}}(0, T), 
}
with  the compatibility conditions $u_0(0) = g_0(0)$ and $u_0(\ell) = h_0(0)$, is locally well-posed in the sence of Hadamard, namely, for an appropriate lifespan $T^*>0$ that depends on the size of the data, it admits a unique solution in the space $C_t([0, T^*]; H_x^s(0, \ell))$  which depends continuously on the data.
\end{theorem}

\begin{proof}
Let $S\big[u_0, g_0, h_0; f\big]$ denote the solution to the full linear problem~\eqref{u-fi-ibvp} on the finite interval. Then, thanks to Theorem \ref{heat-t}, we have the decomposition
\eee{\label{dec-p}
S\big[u_0, g_0, h_0; f\big]  
=
U|_{(0, \ell)} + v|_{(0, \ell)} + w|_{(0, \ell)},
} 
where $U$ is the solution to the half-line problem \eqref{U-hl-ibvp} and $v, w$ are the solutions to the half-line problems \eqref{vw-hl-ibvp-i} with boundary data $a, b$ satisfying \eqref{ab-cond} where $g$ is given in \eqref{q-fi-ibvp} (such data do exist thanks to Theorem \ref{heat-t}). 
As noted above \eqref{a=0}, each of the three half-line problems on the right side of \eqref{dec-p} have been estimated in the  space $X_T = C_t([0, T]; H_x^s(0, \ell))$ (see \cite{hmy2019-rd}). In particular, by the triangle inequality, the decomposition \eqref{dec-p} and the estimates of \cite{hmy2019-rd} imply
\eee{\label{dec-est}
\no{S\big[u_0, g_0, h_0; f\big]}_{X_T}
=
\no{u_0}_{H_x^s(0, \ell)} + \no{g_0}_{H^{\frac{2s+1}{4}}(0, T)} + \no{h_0}_{H^{\frac{2s+1}{4}}(0, T)} + \sqrt T \no{f}_{X_T}, \quad \frac 12 < s < \frac 32. 
} 
This forced linear estimate allows us to establish Theorem \ref{lwp-rd-t} via a fixed point argument after replacing $f = f(u) = |u|^{p-1} u$. In particular, for the iteration map
\eee{
 u \mapsto \Phi(u) := S\big[u_0, g_0, h_0; f(u)\big],
}
estimate \eqref{dec-est} combined with suitable nonlinear estimates from the theory of the \textit{Cauchy problem} (e.g. for $s>\frac 12$ this is just the Sobolev algebra property) implies the existence of a $T^* \in (0, T]$ such that  $\Phi$ is a contraction on an appropriate neighborhood of $X_{T^*}$, thereby resulting in a unique fixed point of $\Phi$ which in turn amounts to a unique solution to the nonlinear problem \eqref{rd-ibvp}.
\end{proof}

\begin{remark}\label{wp-r}
The above proof can be adapted to any setting that allows one to express the linearization of a nonlinear finite interval problem as the sum of half-line problems. In particular, it can be used to deduce the analogue of Theorem \ref{lwp-rd-t} in the low regularity setting of $\frac 12 - \frac 1p < s < \frac 12$ with $p=2,3,4,\ldots$ (namely, to deduce Theorem 5.1 of\cite{hmy2019-rd}) as well as in the even lower range  $0\leq s < \frac 12 - \frac 1p$ addressed for certain nonlinearities in~\cite{o2026}. Moreover, thanks to Theorem~\ref{kdv-t} (which is proved in Section~\ref{kdv-s}), the same reasoning can be employed for deducing corresponding well-posedness results for the KdV equation on a finite interval.  
\end{remark}

We now proceed to the proof of Theorem \ref{heat-t}, which will be accomplished in several steps by extracting an integral equation for $a$ through the combination of \eqref{ab-cond} with the solution formula obtained for the positive half-line problem in \eqref{vw-hl-ibvp-i} via the unified transform. 

\subsection{An integral equation for $a(t)$}
\label{iea-ss}

By means of the unified transform, the positive half-line problem in~\eqref{vw-hl-ibvp-i} has been shown to admit the solution formula (see (16) in \cite{f2008})
\begin{equation}\label{heat-utm-t}
v(x, t) = \frac{1}{i\pi} \int_{k\in\p \mathcal D} e^{ikx - k^2t} \, k \, \widetilde a(k^2, t) dk,
\end{equation}
where $\p \mathcal D$ is the positively oriented boundary of the region 
\eee{\label{d-def}
\mathcal D:= \left\{k\in\mathbb C: \text{Im}(k) > 0 \text{ and } \text{Re}(k^2)<0\right\} \equiv \left\{k \in \mathbb C: \frac \pi 4 < \arg(k) < \frac{3\pi}{4}\right\}
}
depicted in Figure \ref{d-heat-f} and
\begin{equation}\label{atil-def}
\widetilde a(k^2, t) 
:=
\int_{z=0}^t e^{k^2 z} a(z) dz.
\end{equation}
The same formula but with $b$ in place of $a$ gives the solution to the positive half-line problem \eqref{w-hl-ibvp}. Thus, recalling that $w(x, t) = \breve w(\ell-x, t)$, we deduce the corresponding solution formula  for the negative half-line problem in \eqref{vw-hl-ibvp-i} as
\begin{equation}\label{w-utm-t}
w(x, t) = \frac{1}{i\pi} \int_{k\in\p \mathcal D} e^{ik(\ell-x) - k^2t} \, k \, \widetilde b(k^2, t) dk.
\end{equation}

\begin{figure}[ht]
\centering
\vspace{1.5cm}
\begin{tikzpicture}[scale=1]
\pgflowlevelsynccm
\draw[line width=.5pt, black, dotted](1.8,0)--(-1.8,0);
\draw[line width=.5pt, black, dotted](0,0.7)--(0,-0.5);
\draw[line width=.5pt, black, dotted](0,1.2)--(0,1.7);
\draw[line width=.5pt, black](1.5,1.8)--(1.5,1.5);
\draw[line width=.5pt, black](1.5,1.5)--(1.8,1.5);
\node[] at (1.59, 1.685) {\fontsize{8}{8} $k$};
\node[] at (-0.23, -0.23) {\fontsize{10}{10} $0$};
\draw[middlearrow={Stealth[scale=1.3, reversed]}] (0,0) -- (135:1.8);
\draw[middlearrow={Stealth[scale=1.3]}] (0,0) -- (45:1.8);
\node[] at (0, 0.98) {\fontsize{10}{10}\it $\mathcal D$};
\end{tikzpicture}
\vspace{7mm}
\caption{The region $\mathcal D$ and its positively oriented boundary $\p \mathcal D$.}
\label{d-heat-f}
\end{figure}
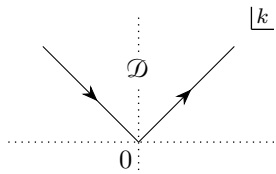

The formulae \eqref{heat-utm-t} and \eqref{w-utm-t} combined with the conditions \eqref{ab-cond} yield the system of integral equations
\begin{align}
a(t) &= g(t) - \frac{1}{i\pi} \int_{k\in\p \mathcal D} e^{ik\ell - k^2t} \, k \, \widetilde b(k^2, t) dk,
\label{a-sys}
\\
b(t) &= -\frac{1}{i\pi} \int_{k\in\p \mathcal D} e^{ik\ell - k^2t} \, k \, \widetilde a(k^2, t) dk.
\label{b-sys}
\end{align}
Due to the presence of $g$ on the right side of \eqref{a-sys}, it is more convenient to eliminate $b$ in favor of an integral equation for $a$. 
Specifically, substituting for  $\widetilde b$ in \eqref{a-sys} via  \eqref{atil-def} and \eqref{b-sys} and using Fubini's theorem, we have
\begin{equation}
a(t) 
= g(t) - \frac{1}{\pi^2}  \int_{z=0}^t  \left(\int_{k\in\p \mathcal D} e^{ik\ell - k^2(t-z)} \, k dk\right)
\int_{r=0}^t a(r) 
\left(
\int_{\lambda\in\p \mathcal D} e^{i\lambda\ell - \lambda^2(z-r)} \, \lambda \, d\lambda 
\right)
dr \, dz.
\label{a-temp}
\end{equation}
Hence, the conditions \eqref{ab-cond} are equivalent to the system \eqref{b-sys}-\eqref{a-temp}. Therefore, proving that there exist functions $a, b$ satisfying \eqref{ab-cond} amounts to proving that the integral equation \eqref{a-temp} for $a$ has a solution and then defining $b$ in terms of that solution via \eqref{b-sys}. En route to accomplishing this task, we establish two results concerning the two complex integrals on the right side of \eqref{a-temp}.

\begin{lemma}\label{j-l}
Suppose $\ell > 0$, $\sigma \leq 0$ and $C_R = \left\{Re^{i\theta}: \frac \pi 4 \leq \theta \leq \frac{3\pi}{4} \right\}$. Then, 
$$
\int_{\lambda\in\p \mathcal D} e^{i\lambda\ell - \lambda^2 \sigma} \, \lambda \, d\lambda
=
\lim_{R\to\infty} \int_{\lambda\in C_R} e^{i\lambda\ell - \lambda^2 \sigma} \, \lambda \, d\lambda = 0.
$$
\end{lemma}

\begin{proof}
The first equality follows directly from the definition of the region $\mathcal D$ via analyticity and Cauchy's theorem. Hence, we only need to establish the second equality. 
This is straightforward since, parametrizing along $C_R$ 
and noting that for $\frac \pi4 \leq \theta \leq \frac{3\pi}{4}$ we have $\frac{1}{\sqrt 2} \leq \sin\theta \leq 1$ and $-1 \leq \cos(2\theta) \leq 0$, we have
\eee{
\left|
\int_{\lambda\in C_R} e^{i\lambda\ell - \lambda^2 \sigma} \, \lambda \, d\lambda
\right|
\leq 
R^2 \int_{\theta=\frac \pi4}^{\frac{3\pi}{4}} e^{-R \ell \sin\theta - R^2 \sigma \cos(2\theta) }  d\theta
\leq
R^2 \int_{\theta=\frac \pi4}^{\frac{3\pi}{4}} e^{-\frac{R \ell}{\sqrt 2}}  d\theta = 
\frac \pi 2 \, e^{-\frac{R \ell}{\sqrt 2}} \, R^2 \to 0, \quad R\to\infty,
\nn
}
as desired, due to the fact that $\ell>0$. Note that we only require $\sigma \leq 0$ (as opposed to $\sigma<0$) because we do not rely on the decay of the time exponential inside $\mathcal D$, as this feature is lost at the boundary $\p \mathcal D$. 
\end{proof}

Lemma \ref{j-l} takes care of the $\lambda$-integral in \eqref{a-temp} when $z-r\leq 0$. The case of $z-r>0$, which is relevant for both the $\lambda$-integral and the $k$-integral in \eqref{a-temp}, will be handled via the following result.
\begin{lemma}\label{j2-l}
Suppose $\ell>0$,  $\sigma>0$. Then, 
$$
\int_{\lambda\in \p \mathcal D} e^{i\lambda\ell - \lambda^2 \sigma} \, \lambda \, d\lambda 
=
\frac{i\sqrt \pi \, \ell e^{-\frac{\ell^2}{4\sigma}}}{2\sigma^{\frac 32}}.
$$
\end{lemma}

\begin{proof}
It suffices to show that
\eee{\label{j2-eq}
\int_{\lambda\in \p \mathcal D} e^{i\lambda\ell - \lambda^2 \sigma} \, \lambda \, d\lambda
=
\int_{\lambda\in \mathbb R} e^{i\lambda\ell - \lambda^2 \sigma} \, \lambda \, d\lambda
}
since then the desired result follows from the fact that $\mathcal F\big\{e^{- \lambda^2 \sigma} \, \lambda\big\}(x) = \frac{\sqrt \pi \, x e^{-\frac{x^2}{4\sigma}}}{2i\sigma^{\frac 32}}$ for any $\sigma>0$.
To prove \eqref{j2-eq}, we note that by analyticity and Cauchy's theorem
$$
\int_{\lambda\in \p \mathcal D} e^{i\lambda\ell - \lambda^2 \sigma} \, \lambda \, d\lambda
=
\int_{\lambda\in \mathbb R} e^{i\lambda\ell - \lambda^2 \sigma} \, \lambda \, d\lambda
-
\lim_{R\to\infty}
\int_{\lambda\in \widetilde C_R} e^{i\lambda\ell - \lambda^2 \sigma} \, \lambda \, d\lambda
$$
where $\widetilde C_R = \left\{Re^{i\theta}: \theta \in \left[0, \frac \pi 4\right] \cup \left[\frac{3\pi}{4},  \pi\right] \right\}$.
We have
\begin{align*}
\left|
\int_{\lambda\in \widetilde C_R} e^{i\lambda\ell - \lambda^2 \sigma} \, \lambda \, d\lambda
\right|
&\leq
R^2 \left( \int_{\theta=0}^{\frac{\pi}{4}} + \int_{\theta=\frac{3\pi}{4}}^\pi \right)
e^{-R \ell \sin\theta - R^2 \sigma \cos(2\theta)} \, d\theta
\nn\\
&=
2R^2  \int_{\theta=0}^{\frac{\pi}{8}}  e^{-R \ell \sin\theta - R^2 \sigma \cos(2\theta)} \, d\theta
+
2R^2  \int_{\theta=\frac{\pi}{8}}^{\frac \pi 4}  e^{-R \ell \sin\theta - R^2 \sigma \cos(2\theta)} \, d\theta.
\end{align*}
For the first integral, we use that fact that $\sin\theta \geq 0$, $\cos(2\theta) \geq \frac{1}{\sqrt 2}$ and $\sigma>0$ to infer
$$
R^2 \int_{\theta=0}^{\frac{\pi}{8}}  e^{-R \ell \sin\theta - R^2 \sigma \cos(2\theta)} \, d\theta
\leq
R^2 \int_{\theta=0}^{\frac{\pi}{8}}  e^{-\frac{R^2 \sigma}{\sqrt 2}} \, d\theta
=
R^2 e^{-\frac{R^2 \sigma}{\sqrt 2}} \, \frac{\pi}{8} \to 0, \quad R \to \infty.
$$
For the second integral, we observe that $\sin\theta \geq \sin(\frac \pi 8) > 0$,  $\cos(2\theta) \geq 0$ and $\ell>0$ to deduce
$$
R^2  \int_{\theta=\frac{\pi}{8}}^{\frac \pi 4}  e^{-R \ell \sin\theta - R^2 \sigma \cos(2\theta)} \, d\theta
\leq
R^2  \int_{\theta=\frac{\pi}{8}}^{\frac \pi 4}  e^{-R \ell \sin(\frac \pi 8)} \, d\theta
=
R^2 e^{-R \ell \sin(\frac \pi 8)} \, \frac \pi 8 \to 0, \ R \to \infty.
$$
Thus, we overall conclude that $\lim_{R\to\infty} \int_{\lambda\in \widetilde C_R} e^{i\lambda\ell - \lambda^2 \sigma} \, \lambda \, d\lambda = 0$,
which implies \eqref{j2-eq}.
\end{proof}

Define  the function 
\begin{equation}\label{ldef}
\Lambda_\ell(\sigma) :=
\left\{
\begin{array}{ll}
0, &\sigma \leq 0,
\\
\dfrac{e^{-\frac{\ell^2}{4\sigma}}}{\sigma^{\frac 32}},
&\sigma>0,
\end{array}
\right.
\end{equation}
which is continuous since for any $\ell>0$ we have $\displaystyle \lim_{\sigma\to 0^+} \Lambda_\ell(\sigma) = 0$. 
Then, combining Lemmas \ref{j-l} and \ref{j2-l}, 
\begin{equation}\label{dp-int}
\int_{\lambda\in\p \mathcal D} e^{i\lambda\ell - \lambda^2 \sigma} \, \lambda \, d\lambda
=
\frac{i\sqrt \pi \, \ell}{2}\Lambda_\ell(\sigma),
\quad \sigma \in \mathbb R.
\end{equation}
Hence,  \eqref{a-temp} becomes
\begin{equation}\label{a-int-eq}
a(t) = g(t) 
+ \frac{\ell^2}{4\pi} 
\int_{z=0}^t 
\Lambda_\ell(t-z)
\int_{r=0}^z  
\Lambda_\ell(z-r)
\, a(r) \, dr \, dz
\end{equation}
where the upper limit of the $r$-integral has changed from $t$ to $z$ in view of the fact that $\Lambda_\ell(\sigma) = 0$ for $\sigma \leq 0$.

\subsection{Existence in $L^2(0, T)$}
\label{cml2-ss}

We now use a contraction mapping argument to show that, given $T>0$ sufficiently small, the integral equation~\eqref{a-int-eq} possesses a solution in $L^2(0, T)$. In turn, this proves that the system \eqref{b-sys}-\eqref{a-temp} and, equivalently, the original conditions \eqref{ab-cond} admit a pair of solutions $a, b \in L^2(0, T)$. 

Consider the map  
\begin{equation}\label{phia-def}
a(t) \mapsto \Phi_g[a](t) 
:=
g(t) 
+ 
\frac{\ell^2}{4\pi} 
\int_{z=0}^t 
\Lambda_\ell(t-z)
\int_{r=0}^z  
\Lambda_\ell(z-r) \, a(r) \, dr \, dz.
\end{equation}
By the triangle inequality, 
\begin{align}
\no{\Phi_g[a]}_{L^2(0, T)} 
&\leq
\no{g}_{L^2(0, T)} 
+
\frac{\ell^2}{4\pi} 
\left(
\int_{t=0}^T
\left|
\int_{z=0}^t 
\Lambda_\ell(t-z)
\int_{r=0}^z  
\Lambda_\ell(z-r)
\, a(r) \, dr \, dz
\right|^2 dt
\right)^{\frac 12}.
\label{cest1}
\end{align}
In view of the inequality 
$
0< 
\sigma^{-\frac 32} e^{-\frac{\ell^2}{4\sigma}}
\leq
\left(\frac{6}{e \ell^2} \right)^{\frac 32}$, $\sigma>0$,
we have
\begin{equation}\label{L-bound}
0\leq \Lambda_\ell(\sigma) \leq \left(\frac{6}{e \ell^2} \right)^{\frac 32}, \quad \sigma \in \mathbb R.
\end{equation}
Thus, continuing from \eqref{cest1} we find 
\ddd{
\no{\Phi_g[a]}_{L^2(0, T)} 
&\leq
\no{g}_{L^2(0, T)} 
+
\frac{\ell^2}{4\pi} 
\left(\frac{6}{e \ell^2} \right)^3
\left(
\int_{t=0}^T
\left(
\int_{z=0}^t 
\int_{r=0}^z  
\left|a(r)\right| dr \, dz
\right)^2 dt
\right)^{\frac 12}
\nn\\
&\leq
\no{g}_{L^2(0, T)} 
+
\frac{\ell^2}{4\pi} 
\left(\frac{6}{e \ell^2} \right)^3
\left(
\int_{t=0}^T
\left(
\int_{z=0}^t
\int_{r=0}^T  
\left|a(r)\right| dr \, dz
\right)^2 dt
\right)^{\frac 12}
\nn\\
&=
\no{g}_{L^2(0, T)} 
+
\frac{\ell^2}{4\pi} 
\left(\frac{6}{e \ell^2} \right)^3
\frac{T^{\frac 32}}{\sqrt 3} \no{a}_{L^1(0, T)}
}
so by the Cauchy-Schwarz inequality
\eee{\label{Phia-l2-est}
\no{\Phi_g[a]}_{L^2(0, T)} 
\leq
\no{g}_{L^2(0, T)} 
+
\frac{18 \sqrt 3 \, T^2}{\pi e^3 \ell^4} \no{a}_{L^2(0, T)}.
}

Let $B(0, \rho) \subset L^2(0, T)$ denote the closed ball of radius $\rho = 2\no{g}_{L^2(0, T)}$ centered at zero. Then, for any $a \in B(0, \rho)$ we have
$$
\no{\Phi_g[a]}_{L^2(0, T)}  \leq \frac \rho 2 + \frac{18 \sqrt 3 \, T^2}{\pi e^3 \ell^4} 
\rho 
=
\left(\frac 12 + \frac{18 \sqrt 3 \, T^2}{\pi e^3 \ell^4} \right)\rho 
$$
so if $T>0$ is such that 
\eee{\label{T-L2}
\frac{18 \sqrt 3 \,T^2}{\pi e^3 \ell^4} \leq \frac 12
\ \Leftrightarrow \
\frac{T}{\ell^2} \leq \frac{\sqrt \pi e^{\frac 32}}{2 \cdot 3^{\frac 54}} \simeq 1
}
then $\Phi_g[a] \in B(0, \rho)$. 
Moreover,  for any $a_1, a_2 \in B(0, \rho)$ we have
\eee{\label{Phia-l2-contr}
\no{\Phi_g[a_1] - \Phi_g[a_2]}_{L^2(0, T)}  
\leq
\frac{18 \sqrt 3 \, T^2}{\pi e^3 \ell^4} \no{a_1-a_2}_{L^2(0, T)}
}
which in view of \eqref{T-L2}  implies that the map $\Phi_g[a]$ is a contraction on $B(0, \rho)$.
Thus, by Banach's fixed point theorem, $\Phi_g[a]$ has a unique fixed point in $B(0, \rho)$, which amounts to a unique solution of the integral equation~\eqref{a-int-eq} for $a$ in  $B(0, \rho)$.  
Furthermore, having proved the existence of such a solution as a fixed point of $\Phi_g[a]$, we can return to \eqref{Phia-l2-est} and obtain the improved size estimate 
\begin{equation}\label{a-l2-est}
\no{a}_{L^2(0, T)} 
\leq
\frac{1}{1-\frac{18 \sqrt 3 \, T^2}{\pi e^3 \ell^4}} \no{g}_{L^2(0, T)}.
\end{equation}
In summary, if the ratio $T/\ell^2$ is sufficiently small such that \eqref{T-L2} is satisfied, then there exists a unique $a \in B(0, 2\no{g}_{L^2(0, T)})\subset L^2(0, T)$ that solves the integral equation \eqref{a-int-eq}.

\subsection{Existence in $H^1(0, T)$}
\label{hse-ss}
By the definition \eqref{ldef}, 
$
\displaystyle \lim_{\sigma\to 0^+} \tfrac{\Lambda_\ell(\sigma)-\Lambda_\ell(0)}{\sigma-0}
=
\lim_{\sigma\to 0^+} \sigma^{-\frac 52} e^{-\frac{\ell^2}{4\sigma}}
= 0
$
and
$
\displaystyle \lim_{\sigma \to 0^-} \tfrac{\Lambda_\ell(\sigma)-\Lambda_\ell(0)}{\sigma-0}
= 0
$
so that $\Lambda_\ell'(0) = 0$. Hence, $\Lambda_\ell \in C^1(\mathbb R)$ with
\eee{\label{L'}
\Lambda_\ell'(\sigma) 
=
\left\{
\begin{array}{ll}
0, &\sigma\leq0,
\\[1mm]
\dfrac{\ell^2 - 6\sigma}{4\sigma^{\frac 72}} \, e^{-\frac{\ell^2}{4\sigma}},
&\sigma>0.
\end{array}
\right. 
} 
Therefore, differentiating \eqref{phia-def} with respect to $t$ and using the Leibniz integral rule,
we obtain
\begin{align*}
\Phi_g[a]'(t)
&= g'(t) 
+ \frac{\ell^2}{4\pi} 
\Lambda_\ell(0)
\int_{r=0}^t
\Lambda_\ell(t-r) \, a(r) \, dr 
+ \frac{\ell^2}{4\pi} 
\int_{z=0}^t 
\p_t \big(\Lambda_\ell(t-z)\big) 
\int_{r=0}^z  
\Lambda_\ell(z-r) \, a(r) \, dr \, dz
\nn\\
&= g'(t) 
- \frac{\ell^2}{4\pi} 
\int_{z=0}^t 
\p_z \big(\Lambda_\ell(t-z)\big) 
\int_{r=0}^z  
\Lambda_\ell(z-r) \, a(r) \, dr \, dz
\end{align*}
upon observing that  $\p_t \big(\Lambda_\ell(t-z)\big)  = -\p_z \big(\Lambda_\ell(t-z)\big) $.
Integrating by parts with respect to $z$ and proceeding in a similar way as above, we further obtain
\begin{align*}
\Phi_g[a]'(t)
&= g'(t) 
- \frac{\ell^2}{4\pi} 
\left[
\Lambda_\ell(t-z)
\int_{r=0}^z  
\Lambda_\ell(z-r) \, a(r) \, dr
\right]_{z=0}^t 
+ 
\frac{\ell^2}{4\pi} 
\int_{z=0}^t 
\Lambda_\ell(t-z)
\, \p_z 
\left(
\int_{r=0}^z  
\Lambda_\ell(z-r) \, a(r) \, dr
\right) dz
\nn\\
&= g'(t) 
- 
\frac{\ell^2}{4\pi} 
\int_{z=0}^t 
\Lambda_\ell(t-z)
\int_{r=0}^z  
\p_r 
\big( \Lambda_\ell(z-r) \big) \,
a(r) \, dr \, dz.
\end{align*}
Thus, integrating by parts in the $r$-integral, we find
\eee{\label{ibp}
\Phi_g[a]'(t)
= g'(t) 
+ \frac{\ell^2}{4\pi} 
\int_{z=0}^t 
\Lambda_\ell(t-z)
\left[
\Lambda_\ell(z) a(0)
+
\int_{r=0}^z  
\Lambda_\ell(z-r) \, a'(r) \, dr
\right] dz.
}
Therefore, by the condition \eqref{a=0}, which must hold since $H^1(0, T)$ corresponds to $H^{\frac{2s+1}{4}}(0, T)$ with $s=\frac 32$, we conclude that
\begin{equation}\label{a'-int-eq}
\Phi_g[a]'(t)
= g'(t) 
+ \frac{\ell^2}{4\pi} 
\int_{z=0}^t 
\Lambda_\ell(t-z)
\int_{r=0}^z  
\Lambda_\ell(z-r) \, a'(r) \, dr \, dz
\equiv
\Phi_{g'}[a'](t).
\end{equation}
Thanks to this observation,  the previously derived $L^2$-estimate \eqref{Phia-l2-est} with $a', g'$ in place of $a, g$ yields the following estimate for the $L^2$-norm of $\Phi_g[a]'$: 
\begin{equation}\label{Phia'-l2-est}
\no{\Phi_g[a]'}_{L^2(0, T)} 
=
\no{\Phi_{g'}[a']}_{L^2(0, T)} 
\leq
\no{g'}_{L^2(0, T)} 
+
\frac{18 \sqrt 3 \, T^2}{\pi e^3 \ell^4} \no{a'}_{L^2(0, T)}.
\end{equation}

Combining estimates \eqref{Phia-l2-est} and \eqref{Phia'-l2-est} we deduce
\begin{equation}\label{Phia-h1-est}
\no{\Phi_g[a]}_{H^1(0, T)} 
\leq
\no{g}_{H^1(0, T)} 
+
\frac{18 \sqrt 3 \, T^2}{\pi e^3 \ell^4} \no{a}_{H^1(0, T)}.
\end{equation}
Similarly, observing that
\begin{equation}
\Phi_g[a_1]'(t)-\Phi_g[a_2]'(t)
= 
\frac{\ell^2}{4\pi} 
\int_{z=0}^t 
\Lambda_\ell(t-z)
\int_{r=0}^z  
\Lambda_\ell(z-r) \big[a_1'(r)-a_2'(r)\big] \, dr \, dz
\equiv
\Phi_{g'}[a_1'](t) - \Phi_{g'}[a_2'](t)
\end{equation}
and recalling the $L^2$-estimate \eqref{Phia-l2-contr}, we infer the $H^1$-estimate
\eee{\label{Phia-h1-contr}
\no{\Phi_g[a_1]-\Phi_g[a_2]}_{H^1(0, T)}
\leq
\frac{18 \sqrt 3 \, T^2}{\pi e^3 \ell^4} \no{a_1-a_2}_{H^1(0, T)}.
}
Estimates \eqref{Phia-h1-est} and \eqref{Phia-h1-contr} imply that, for any $T>0$ satisfying the condition \eqref{T-L2},  the map $a \mapsto \Phi[a]$ is a contraction in the ball $B(0, \rho) \subset H^1(0, T)$ with $\rho = 2\no{g}_{H^1(0, T)}$. Hence, there is a unique solution to the integral equation \eqref{a-int-eq} in $B(0, \rho)$ satisfying the size estimate
\begin{equation}\label{a-h1-est}
\no{a}_{H^1(0, T)} 
\leq
\frac{1}{1-\frac{18 \sqrt 3 \, T^2}{\pi e^3 \ell^4}} \no{g}_{H^1(0, T)}.
\end{equation}

\begin{remark}
Instead of transferring the derivative from $\Lambda_\ell$ to $a$, as we did via integration by parts in order to arrive at \eqref{ibp}, we could use the fact that $\Lambda_\ell'(\sigma)$ is uniformly bounded with respect to $\sigma$ in order to always work with $a \in L^2(0, T)$. 
However, in order to prove existence of solution as a fixed point in $H^1(0, T)$ via contraction, we would still need to replace the $L^2$-norm by an $H^1$-norm, so this alternative approach does not offer an advantage.
\end{remark}

\subsection{Existence in $H^m(0, T)$ for any  $m\geq 0$}
\label{frac-ss}

Suppose first that $m \in \mathbb N_0$, so that
\eee{\label{sob-int}
\no{\Phi_g[a]}_{H^m(0, T)} = \sum_{n=0}^m \big\| \Phi_g[a]^{(n)} \big\|_{L^2(0, T)}.
}
Generalizing \eqref{L'}, for any $0\leq n \leq m$ we have $\Lambda_\ell \in C^n(\mathbb R)$ with
\eee{\label{Lm}
\Lambda_\ell^{(n)}(\sigma) 
=
\left\{
\begin{array}{ll}
0, &\sigma\leq0,
\\
\dfrac{p_n(\sigma)}{q_n(\sigma)} \, e^{-\frac{\ell^2}{4\sigma}},
&\sigma>0,
\end{array}
\right.
} 
for some polynomials $p_n$, $q_n$ with $\text{deg}(p_n) < \text{deg}(q_n)$. Thus, we can employ the Leibniz integral rule repeatedly to establish, via induction, the following generalization of formula \eqref{a'-int-eq}:
\begin{equation}\label{aj-int-eq}
\Phi_g[a]^{(n)}(t)
= 
g^{(n)}(t) + \frac{\ell^2}{4\pi} 
\int_{z=0}^t 
\Lambda_\ell(t-z)
\int_{r=0}^z  
\Lambda_\ell(z-r) \, a^{(n)}(r) \, dr \, dz
\equiv
\Phi_{g^{(n)}}[a^{(n)}](t),
\end{equation}
where we have also used the fact that, by \eqref{an=0},   $a^{(n)}(0) = 0$ for all $0\leq n \leq m-1$ 
(note that \eqref{an=0} holds for all integers $0\leq n < \frac{2s-1}{4} = m - \frac 12$).
In turn, estimate \eqref{Phia-l2-est} with $a^{(n)}, g^{(n)}$ in place of $a, g$ yields
\begin{equation}\label{Phiaj-l2-est}
\big\|\Phi_g[a]^{(n)}\|_{L^2(0, T)} 
=
\big\|\Phi_{g^{(n)}}[a^{(n)}]\|_{L^2(0, T)} 
\leq
\big\|g^{(n)}\|_{L^2(0, T)} 
+
\frac{18 \sqrt 3 \, T^2}{\pi e^3 \ell^4} 
\big\|a^{(n)}\|_{L^2(0, T)}, 
\quad
n \in \mathbb N_0.
\end{equation}
Hence, in view of \eqref{sob-int} we infer
\begin{equation}\label{Phia-hj-est}
\no{\Phi_g[a]}_{H^m(0, T)} 
\leq
\no{g}_{H^m(0, T)} 
+
\frac{18 \sqrt 3 \, T^2}{\pi e^3 \ell^4} 
\no{a}_{H^m(0, T)}, \quad m \in \mathbb N_0.
\end{equation}

The case of general $m\geq 0$ can be handled by combining the integer estimate \eqref{Phia-hj-est} with the following fundamental interpolation result.
\\[2mm]
\textbf{Theorem} (Theorem 5.1 in \cite{lm1972}).
{\itshape Let $\{X, Y\}$ be a couple of Hilbert spaces and let $\{\mathcal X, \mathcal Y\}$ be a second couple of Hilbert spaces with properties analogous to the first one. Let $\mathcal M$ be a continuous linear operator of $X$ into $\mathcal X$ and of $Y$ into $\mathcal Y$, i.e. $\mathcal M \in \mathscr L(X; \mathcal X) \cap \mathscr L(Y; \mathcal Y)$. Then, 
\begin{equation}\label{inter-t}
\mathcal M \in \mathscr L\big(\left[X, Y\right]_\beta; \left[\mathcal X, \mathcal Y\right]_\beta\big) \quad \forall \beta\in (0, 1).
\end{equation}
}

\vskip 2mm

The operator $a(t) \mapsto \mathcal M[a](t) := \Phi_g[a](t) - g(t)$ is linear in $a$. Moreover, the calculations leading to estimate~\eqref{Phia-hj-est} can be easily modified to yield
\begin{equation}
\no{\Phi_g[a] - g}_{H^m(0, T)} 
\leq
\frac{18 \sqrt 3 \, T^2}{\pi e^3 \ell^4} 
\no{a}_{H^m(0, T)}, \quad m \in \mathbb N_0.
\end{equation}
Thus, in particular, for any $m = \left\lfloor m \right\rfloor + \beta$ with $\beta \in (0, 1)$, we have the estimates
\eee{
\no{\mathcal M[a]}_{H^{\lfloor m \rfloor}(0, T)} 
\leq
\frac{18 \sqrt 3 \, T^2}{\pi e^3 \ell^4} \no{a}_{H^{\lfloor m \rfloor}(0, T)},
\quad
\no{\mathcal M[a]}_{H^{\lfloor m \rfloor+1}(0, T)}
\leq
\frac{18 \sqrt 3 \, T^2}{\pi e^3 \ell^4} \no{a}_{H^{\lfloor m \rfloor+1}(0, T)}
\nn
}
which imply that $\mathcal M[a]$ is continuous  from $X = H^{\lfloor m \rfloor}(0, T)$ to $\mathcal X = H^{\lfloor m \rfloor}(0, T)$ and also 
 from $Y=H^{\lfloor m \rfloor+1}(0, T)$ to $\mathcal Y = H^{\lfloor m \rfloor+1}(0, T)$. Therefore, by the interpolation \eqref{inter-t}, $\mathcal M[a]$ is continuous (equivalently, bounded)  from  $\left[X, Y\right]_\beta = \left[H^{\lfloor m \rfloor}(0, T), H^{\lfloor m \rfloor+1}(0, T)\right]_\beta = H^m(0, T)$ to $\left[\mathcal X, \mathcal Y\right]_\beta = \left[H^{\lfloor m \rfloor}(0, T), H^{\lfloor m \rfloor+1}(0, T)\right]_\beta = H^m(0, T)$ and, consequently,
\begin{equation}
\left\| \mathcal M[a] \right\|_{H^m(0, T)}
\equiv
\no{\Phi_g[a]-g}_{H^m(0, T)}
\leq
\frac{18 \sqrt 3 \, T^2}{\pi e^3 \ell^4}
\left\| a \right\|_{H^m(0, T)}, \quad m \geq 0.
\end{equation}
Hence, since by the reverse triangle inequality $\no{\Phi_g[a]}_{H^m(0, T)} - \no{g}_{H^m(0, T)} \leq \no{\Phi_g[a]-g}_{H^m(0, T)}$, we obtain
\begin{equation}
\no{\Phi_g[a]}_{H^m(0, T)}
\leq
\no{g}_{H^m(0, T)}
+
\frac{18 \sqrt 3 \, T^2}{\pi e^3 \ell^4}
\left\| a \right\|_{H^m(0, T)}, \quad m \geq 0.
\end{equation}
This estimate is entirely analogous to \eqref{Phia-l2-est}. Furthermore, adjusting its derivation accordingly, we readily obtain the analogue of the contraction inequality \eqref{Phia-l2-contr}. Thus, for any $m\geq 0$ and $T>0$ satisfying \eqref{T-L2}, there is a unique $a \in B(0, \rho) \subset H^m(0, T)$ with $\rho = 2\no{g}_{H^m(0, T)}$ that satisfies the integral equation \eqref{a-int-eq} and admits the size estimate
\begin{equation}\label{a-hm-est}
\no{a}_{H^m(0, T)} 
\leq
\frac{1}{1-\frac{18 \sqrt 3 \, T^2}{\pi e^3 \ell^4}} \no{g}_{H^m(0, T)}, \quad m \geq 0.
\end{equation}

\begin{remark}
Instead of the above interpolation argument, one can proceed directly by  estimating the fractional Sobolev-Slobodeckij-Gagliardo seminorm. 
\end{remark}

\section{The Linearized K{\normalfont d}V Equation}
\label{kdv-s}

We turn our attention to dispersive equations and, specifically, the following initial-boundary value problem for the linearized KdV equation on a finite interval:
\begin{equation}\label{fi-kdv}
\begin{aligned}
&u_t + u_x + u_{xxx} = f(x, t), \quad x \in (0, \ell), \ t \in (0, T),
\\
&u(x, 0) = u_0(x), \quad x \in (0, \ell),
\\
&u(0, t) = g_0(t), \quad  u(\ell, t) = h_0(t), \quad  u_x(\ell, t) = h_1(t), \quad t \in (0, T).
\end{aligned}
\end{equation}
Note that the positive sign of the third derivative term implies the need for one boundary condition on the left and two on the right (a negative sign would require two boundary conditions on the left and one on the right --- see also problem \eqref{neg-hl-kdv-y} below). 

In addition, analogously to the heat equation, extending $u_0$ and $f$ from the finite interval $(0, \ell)$ to the negative half-line $(-\infty, \ell)$ via suitable functions $U_0$ and $F$ to be discussed later, we consider the negative half-line problem 
\begin{equation}\label{neg-hl-kdv} 
\begin{aligned}
&U_t + U_x + U_{xxx} = F(x, t), \quad x \in (-\infty, \ell), \ t \in (0, T),
\\
&U(x, 0) = U_0(x), \quad x \in (-\infty, \ell),
\\
&U(\ell, t) = h_0(t), \quad U_x(\ell, t) = h_1(t), \quad t \in (0, T).
\end{aligned}
\end{equation}
The need for prescribing two pieces of boundary data at $x=\ell$ in problem \eqref{neg-hl-kdv} can be appreciated by observing that the function $\breve U(x, t) := U(\ell-x, t)$ satisfies the positive half-line problem
\begin{equation}\label{neg-hl-kdv-y} 
\begin{aligned}
&\breve U_t - \breve U_x - \breve U_{xxx} = \breve F(x, t) := F(\ell-x, t), \quad x \in (0, \infty), \ t \in (0, T),
\\
&\breve U(x, 0) = \breve U_0(x) := U_0(\ell-x), \quad x \in (0, \infty),
\\
&\breve U(0, t) = h_0(t), \quad \breve U_x(0, t) = -h_1(t), \quad t \in (0, T),
\end{aligned}
\end{equation}
which involves the negative-sign linearized KdV equation and  requires two boundary conditions at $x=0$  \cite{f2008}. 

The negative-sign linearized KdV equation in problem \eqref{neg-hl-kdv-y} is a special case of the so-called higher-order nonlinear Schr\"odinger equation $i u_t + i \beta u_{xxx} + \alpha u_{xx} + i\delta u_x = \kappa |u|^{\lambda-1} u$ after setting $\alpha=\kappa=0$, $\beta=\delta=-1$. The analogue of the half-line problem \eqref{neg-hl-kdv-y} for this more general equation is studied in  \cite{amo2026} and allows us to reduce the estimation of the  interval problem \eqref{fi-kdv} to that of the following reduced interval problem  for the difference $q:=u-U|_{x\in(0, \ell)}$:
\begin{equation}\label{fi-kdv-red}
\begin{aligned}
&q_t + q_x + q_{xxx} = 0, \quad x \in (0, \ell), \ t \in (0, T),
\\
&q(x, 0) = 0, \quad x \in (0, \ell),
\\
&q(0, t) = g(t) := g_0(t) - U(0, t), \quad q(\ell, t) = 0, \quad q_x(\ell, t) = 0, \quad t \in (0, T).
\end{aligned}
\end{equation}
Note that problem \eqref{fi-kdv-red} is precisely problem \eqref{fi-kdv-red-i} with the specific choice of $g$ given above. 
Therefore, formally, the solution to problem \eqref{fi-kdv-red} can be expressed as the sum \eqref{sup-kdv-i} of the restrictions on $(0, \ell)$ of the solutions $v(x, t)$ and $w(x, t)$ to the two half-line problems in \eqref{vw-hllr-i} provided that the relevant boundary data $a(t)$ and $b(t), c(t)$ satisfy the conditions (for rough data, these conditions are understood to hold almost everywhere) 
\begin{equation}\label{abc-cond}
\begin{aligned}
a(t) &= g(t) - w(0, t),
\\
b(t) &= -v(\ell, t),
\quad
c(t) = -v_x(\ell, t).
\end{aligned}
\end{equation}
Note that the function $\breve w(x, t) := w(\ell-x, t)$ satisfies a positive half-line problem for the negative-sign linearized KdV equation, namely
\begin{equation}\label{hlr-kdv-trans}
\begin{aligned}
&\breve w_t - \breve w_x - \breve w_{xxx} = 0, \quad x \in (0,\infty), \ t \in (0, T),
\\
&\breve w(x, 0) = 0, \quad x \in (0, \infty),
\\
&\breve w(0, t) = b(t), \quad  \breve w_x(0, t) = -c(t), \quad t \in (0, T),
\end{aligned}
\end{equation}
which is of the same type as problem \eqref{neg-hl-kdv-y} above and hence, as previously noted, can be estimated via the results of \cite{amo2026}.

In view of the decomposition \eqref{sup-kdv-i}, the reduced interval problem~\eqref{fi-kdv-red} --- and, in turn, the full interval problem~\eqref{fi-kdv}  --- can be estimated through the corresponding analysis of positive half-line problems, specifically the first problem in  \eqref{vw-hllr-i} and the problems \eqref{neg-hl-kdv-y} and \eqref{hlr-kdv-trans}.  Therefore, as in the case of the heat equation, 
\textit{the  study of the linearized KdV equation and its nonlinear counterparts on the finite interval can be carried through the corresponding analysis on the half-line.}
However, in order for the formal decomposition \eqref{sup-kdv-i} to be made rigorous, one must prove that there indeed exist functions $a, b$ satisfying the required conditions \eqref{abc-cond}.

In this connection, we mention that the positive half-line problem in \eqref{vw-hllr-i}, which involves the standard (positive-sign) linearized KdV equation, was studied in the works \cite{bsz2002, ck2002, h2005, fhm2016} and was recently revisited in~\cite{amo2024} in a more general context. On the other hand, the problems \eqref{neg-hl-kdv-y} and~\eqref{hlr-kdv-trans} for the negative-sign KdV equation were considered in \cite{h2006,fam2007} and, very recently, in \cite{amo2026}.

Specifically, according to the linear theory of \cite{amo2024}, for $a \in H^{\frac{s+1}{3}}(0, T)$ with $s \geq 0$ and $s\neq \frac 12$, the solution $v$ of the positive-sign linearized KdV problem in \eqref{vw-hllr-i} belongs to the space $C_t([0, T]; H_x^s(0, \infty)) \cap C_x([0, \infty); H_t^{\frac{s+1}{3}}(0, T))$. 
Furthermore, according to the linear theory of \cite{amo2026}, if $0 \leq s \leq 2$ with $s\neq \frac 12$ then for $\breve U_0 \in H^s(0, \infty)$, $h_0 \in H^{\frac{s+1}{3}}(0, T)$, $h_1 \in H^{\frac{s}{3}}(0, T)$  and $\breve F \in C_t([0, T]; H_x^s(0, \infty))$ the solution $\breve U$ to the negative-sign linearized KdV problem~\eqref{neg-hl-kdv-y} belongs to the space $C_t([0, T]; H_x^s(0, \infty)) \cap C_x([0, \infty); H_t^{\frac{s+1}{3}}(0, T))$ with $\breve U_x \in C_x([0, \infty); H_t^{\frac{s}{3}}(0, T))$. This result also applies to the solution $\breve w$ of problem~\eqref{hlr-kdv-trans}, which is a special case of problem \eqref{neg-hl-kdv-y}, this time without the need for an upper bound on $s$, i.e. for all $s \geq 0$ with $s\neq \frac 12$. 

As a consequence of the above results, for the interval problem~\eqref{fi-kdv} considered here, we assume that $s\geq 0$ and take $u_0 \in H^s(0, \ell)$, $g_0, h_0 \in H^{\frac{s+1}{3}}(0, T)$, $h_1 \in H^{\frac s3}(0, T)$ and $f \in C_t([0, T]; H_x^s(0, \ell))$. 
Moreover, by the Sobolev embedding theorem, the data must also satisfy the compatibility conditions $u_0(0) = g_0(0)$ and $u_0(\ell) = h_0(0)$ for $s>\frac 12$, and $u_0'(\ell) = h_1(0)$ for $s>\frac 32$. Hence, since Sobolev functions in $H^s$ with $s>\frac 12$ are continuous, we deduce the following condition for the nonzero boundary datum of problem \eqref{fi-kdv-red}:
\eee{\label{g=0-kdv}
g(0) \equiv g_0(0) - U(0, 0) = u_0(0) - U_0(0) = 0, \quad s>\frac 12.
}
Additionally, by the regularity of $v, w$ stated above and the existence of traces in $H^s$ when $s>\frac 12$, the conditions~\eqref{abc-cond} and \eqref{g=0-kdv} imply that the boundary data $a, b, c$ of the two problems in \eqref{vw-hllr-i}   must satisfy  
\eee{\label{a=0-kdv}
\begin{aligned}
&a(0) \equiv v(0, 0) = v(x, 0)\big|_{x=0} = 0, 
\\
&b(0) \equiv w(\ell, 0)= w(x, 0)\big|_{x=\ell} = 0,
\end{aligned}
\quad s>\frac 12,
\qquad
c(0) \equiv w_x(\ell, 0) =  w_x(x, 0)\big|_{x=\ell} = 0, \quad s>\frac 32.
}
Note that these conditions are consistent with \eqref{abc-cond} in view of \eqref{g=0-kdv}. 

More generally, for any $n\in\mathbb N$ and $s>3n+\frac 12$, which amounts to $\frac{s+1}{3} > n + \frac 12$, the regularity of the initial and boundary data implies the existence of traces such that, by means of \eqref{fi-kdv-red}, 
\eee{\label{gn=0-kdv}
g^{(n)}(0) \equiv \p_t^n q(0, 0) 
= \p_t^n q(x, t)\big|_{x=t=0}
= -\left(\p_x+\p_x^3\right)^n q(x, t)\big|_{x=t=0}
= -\left(\p_x+\p_x^3\right)^n q(x, 0)\big|_{x=0}
= 0.
}
Similarly, for the positive half-line problem in \eqref{vw-hllr-i}, 
\eee{\label{an=0-kdv}
a^{(n)}(0) \equiv \p_t^n v(0, t)\big|_{t=0} =   \p_t^n v(x, t)\big|_{x=t=0}  = 
-\left(\p_x+\p_x^3\right)^n v(x, t)\big|_{x=t=0} =  -\left(\p_x+\p_x^3\right)^n v(x, 0)\big|_{x=0}  = 0, 
}
which is consistent with \eqref{abc-cond} in view of \eqref{gn=0-kdv}. Analogous equalities consistent with \eqref{ab-cond} also hold for $b^{(n)}(0)$ and~$c^{(n)}(0)$. 

Theorem \ref{kdv-t} on the existence of $a, b, c$ that satisfy \eqref{abc-cond} is established below. Before giving the proof, we highlight the significance of Theorem \ref{kdv-t} via the following nonlinear well-posedness result, which corresponds to Theorem~1.2 of~\cite{bsz2003} and, in our case, follows via a straightforward application of Theorem \ref{kdv-t} along with the results of~\cite{bsz2002,amo2024,amo2026} on the \textit{half-line} problems for the positive and negative-sign KdV equations (see Remark \ref{wp-r}):
\begin{theorem}[KdV on a finite interval]\label{lwp-kdv-t}
For $s \geq 0$ with $s\notin \left(3\mathbb N_0+\frac 12\right) \cup \left(3\mathbb N_0+\frac 32\right)$, the finite interval problem for the KdV equation
\ddd{
&u_t + u_x + u_{xxx} + uu_x = 0, \quad x\in (0, \ell), \ t\in (0, T), 
\\
&u(x, 0) = u_0(x) \in H^s(0, \ell),
\\
&u(0, t) = g_0(t) \in H^{\frac{s+1}{3}}(0, T), 
\quad
u(\ell, t) = h_0(t) \in H^{\frac{s+1}{3}}(0, T), 
\quad
u_x(\ell, t) = h_1(t) \in H^{\frac{s}{3}}(0, T), 
}
with suitable compatibility conditions between the initial and boundary data, is locally well-posed in the sence of Hadamard, namely, for an appropriate lifespan $T^*>0$ that depends on the size of the data, it admits a unique solution in the space $C_t([0, T^*]; H_x^s(0, \ell)) \cap L_t^2((0, T^*); H_x^{s+1}(0, \ell))$  which depends continuously on the data.
\end{theorem}

\begin{remark}
The finite interval smoothing effect manifested by the presence of the space $L_t^2((0, T^*); H_x^{s+1}(0, \ell))$ in Theorem \ref{lwp-kdv-t} can be easily extracted from the linear estimate for the spatial derivative of order $\sigma$ of the solution to the linear problem~\eqref{fi-kdv} in the space $L_x^\infty((0, \ell); H_t^{\frac{s+1-\sigma}{3}}(0, T))$, which in turn follows readily from the time estimates obtained in the half-line works \cite{bsz2002,amo2024,amo2026}. 
\end{remark}

We proceed to the proof of Theorem \ref{kdv-t}, which will be accomplished in several steps by extracting an integral equation for $a$ through the combination of \eqref{abc-cond} with the solution formula obtained for the two problems in \eqref{vw-hllr-i}  via the unified transform.

\subsection{An integral equation for $a(t)$}
\label{iea-kdv-ss}

By means of the unified transform, the half-line problems \eqref{vw-hllr-i} have been shown to admit the solution formulae (see (65) in \cite{amo2024} and (2.13) in \cite{amo2026})
\begin{align}
v(x, t) &= \frac{1}{2\pi} \int_{k \in \mathcal C_+} e^{ikx + i (k^3-k) t} \left(1-3k^2\right) \widetilde a(k^3-k, T) dk,
\label{v-kdv-sol}
\\
w(x, t) &= \frac{1}{2\pi} \int_{k \in \mathcal C_-} e^{ik(\ell-x) - i (k^3-k) t} 
\left[ i \left(k-\nu\right) \widetilde c(-(k^3-k), T) +  \left(k^2-\nu^2\right) \widetilde b(-(k^3-k), T)\right] dk,
\label{w-kdv-sol}
\end{align}
where 
\eee{\label{f-til-kdv}
\nu = \nu(k) := -\frac k2 + \left(1-\frac 34 k^2\right)^{\frac 12},
\quad
\widetilde f(k, T) := \int_0^T e^{-i k t} f(t) dt
}
and for 
\ddd{\label{clr-kdv}
&\mathcal C_R := \left\{\gamma(y) := \frac{1}{\sqrt 3} \sqrt{y^2+1} + i y, \ 0 < y<\infty\right\}, 
\quad
\mathcal C_L := \left\{-\overline{\gamma(y)}, \ 0 < y<\infty\right\}
}
the complex contours $\mathcal C_+$, $\mathcal C_-$ are given by
\eee{\label{cpm-kdv}
\mathcal C_+ 
=
\left\{y: |y| < \frac{1}{\sqrt 3}\right\}
\cup
\mathcal C_L 
\cup
\mathcal C_R,
\quad
\mathcal C_- 
=
\left\{y: |y| >  \frac{1}{\sqrt 3} \right\}
\cup
\mathcal C_L
\cup
\mathcal C_R
}
with the orientation shown in Figure \ref{d-kdv-f}.

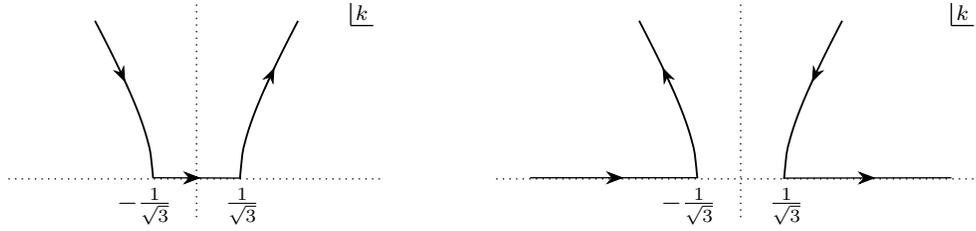
\begin{figure}[ht]
\centering
\vspace{2.5cm}
\begin{tikzpicture}[scale=1]
\pgflowlevelsynccm
\draw[line width=.5pt, black, dotted](-2.5,0.01)--(2.5,0.01);
\draw[line width=.5pt, black, dotted](0,-0.5)--(0, 2.35);
\draw[line width=.5pt, black](2.05,2.35)--(2.05,2.05);
\draw[line width=.5pt, black](2.05,2.05)--(2.35,2.05);
\node[] at (2.14, 2.225) {\fontsize{8}{8} $k$};
\node[] at (-0.75, -0.35) {\fontsize{10}{10} $-\frac{1}{\sqrt 3}$};
\node[] at (0.55, -0.35) {\fontsize{10}{10} $\frac{1}{\sqrt 3}$};
\draw[middlearrow={Stealth[scale=1.3]}, black] (-1,1.41) -- (-53:-1.5);
\draw[middlearrow={Stealth[scale=1.3, reversed]}, black] (1,1.41) -- (53:1.5);
\draw[middlearrow={Stealth[scale=1.3]}, black, line width=.5pt] (-0.59,0.025) -- (2:0.59);
\draw[domain=0.5774:1.35, variable=\x, smooth, black, line width=0.7pt] plot ({\x}, {sqrt(3*\x*\x-1)});
\draw[domain=-1.35:-0.5774, variable=\x, smooth, black, line width=0.7pt] plot ({\x}, {sqrt(3*\x*\x-1)});
\end{tikzpicture}
\hspace*{7cm}
\begin{tikzpicture}[scale=1]
\pgflowlevelsynccm
\draw[line width=.5pt, black, dotted](-3.25,0.01)--(3.25,0.01);
\draw[line width=.5pt, black, dotted](0,-0.5)--(0, 2.35);
\draw[line width=.5pt, black](2.8,2.35)--(2.8,2.05);
\draw[line width=.5pt, black](2.8,2.05)--(3.1,2.05);
\node[] at (2.89, 2.225) {\fontsize{8}{8} $k$};
\node[] at (-0.75, -0.35) {\fontsize{10}{10} $-\frac{1}{\sqrt 3}$};
\node[] at (0.55, -0.35) {\fontsize{10}{10} $\frac{1}{\sqrt 3}$};
\draw[middlearrow={Stealth[scale=1.3, reversed]}, black] (-1,1.41) -- (-53:-1.5);
\draw[middlearrow={Stealth[scale=1.3]}, black] (1,1.41) -- (53:1.5);
\draw[middlearrow={Stealth[scale=1.3, reversed]}, black, line width=.5pt] (-0.565,0.025) -- (-0.5:-2.8);
\draw[middlearrow={Stealth[scale=1.3]}, black, line width=.5pt] (0.565,0.02) -- (0.5:2.8);
\draw[domain=0.5774:1.35, variable=\x, smooth, black, line width=0.7pt] plot ({\x}, {sqrt(3*\x*\x-1)});
\draw[domain=-1.35:-0.5774, variable=\x, smooth, black, line width=0.7pt] plot ({\x}, {sqrt(3*\x*\x-1)});
\end{tikzpicture}
\vspace{7mm}
\caption{The oriented contours $\mathcal C_+$ (left) and $\mathcal C_-$ (right) for the solution formulae \eqref{v-kdv-sol} and~\eqref{w-kdv-sol}, as defined by \eqref{cpm-kdv}.}
\label{d-kdv-f}
\end{figure}

Combining \eqref{abc-cond} with the solution formula \eqref{w-kdv-sol} and the definitions of $\widetilde b, \widetilde c$ according to \eqref{f-til-kdv}, we have
\eee{
a(t) 
=
g(t) - \frac{1}{2\pi} \int_{k \in \mathcal C_-} e^{ik \ell - i (k^3-k) t} 
\left[ i \left(k-\nu\right) \int_{z=0}^T e^{i (k^3-k) z} c(z) dz +  \left(k^2-\nu^2\right) \int_{z=0}^T e^{i (k^3-k) z} b(z) dz \right] dk.
}
Hence, using \eqref{abc-cond} once again to substitute for $b, c$ in terms of $v$ and then employing formula \eqref{v-kdv-sol}, we obtain the integral equation
\aaa{
a(t) 
=
g(t) 
&- \frac{1}{(2\pi)^2} \int_{k \in \mathcal C_-} e^{ik \ell - i (k^3-k) t} 
 \left(k-\nu\right) \int_{z=0}^T e^{i (k^3-k) z} Q_1(z) dz dk
\nn\\
&+ \frac{1}{(2\pi)^2} \int_{k \in \mathcal C_-} e^{ik \ell - i (k^3-k) t} \left(k^2-\nu^2\right) \int_{z=0}^T e^{i (k^3-k) z} Q_2(z)  dk
\label{a-eq-kdv0}
}
where 
\eee{\label{qj-kdv-def}
Q_j(z) 
:= 
\int_{\lambda \in \mathcal C_+} e^{i\lambda\ell + i (\lambda^3-\lambda) z} \lambda^{2-j} \left(1-3\lambda^2\right) \int_{r=0}^T e^{-i(\lambda^3-\lambda)r} a(r) dr d\lambda,
\quad j = 1, 2. 
}

The integral equation  \eqref{a-eq-kdv0} can be written in the form
\eee{\label{a-eq-kdv}
a(t) = \Phi_g[a](t)
}
where, recalling the definition of the contour $\mathcal C_-$, the right side has been rearranged from the one in \eqref{a-eq-kdv0} to 
\aaa{
\Phi_g[a](t)
:=
g(t) &- \frac{1}{(2\pi)^2} \int_{k \in (-\infty, -1) \cup (1, \infty)} e^{ik \ell - i (k^3-k) t} 
 \left(k-\nu\right) \int_{z=0}^T e^{i (k^3-k) z}  Q_1(z) dz dk
\label{ag1}
\\
&
- \frac{1}{(2\pi)^2} \int_{k \in (-1, -\frac{1}{\sqrt 3}) \cup (\frac{1}{\sqrt 3}, 1)} e^{ik \ell - i (k^3-k) t} 
 \left(k-\nu\right) \int_{z=0}^T e^{i (k^3-k) z}  Q_1(z) dz dk
\label{ag2}
\\
&
- \frac{1}{(2\pi)^2} \int_{k\in \mathcal C_L\cup\mathcal C_R} e^{ik \ell - i (k^3-k) t} 
 \left(k-\nu\right) \int_{z=0}^T e^{i (k^3-k) z} Q_1(z) dz dk
\label{ag3}
\\
&
+ \frac{1}{(2\pi)^2} \int_{k \in (-\infty, -1) \cup (1, \infty)} e^{ik \ell - i (k^3-k) t} 
 \left(k^2-\nu^2\right) \int_{z=0}^T e^{i (k^3-k) z}  Q_2(z) dz dk
\label{ag4}
\\
&
+ \frac{1}{(2\pi)^2} \int_{k \in (-1, -\frac{1}{\sqrt 3}) \cup (\frac{1}{\sqrt 3}, 1)} e^{ik \ell - i (k^3-k) t} 
 \left(k^2-\nu^2\right) \int_{z=0}^T e^{i (k^3-k) z}  Q_2(z) dz dk
\label{ag5}
\\
&
+ 
\frac{1}{(2\pi)^2} \int_{k\in \mathcal C_L\cup\mathcal C_R} e^{ik \ell - i (k^3-k) t} \left(k^2-\nu^2\right) \int_{z=0}^T e^{i (k^3-k) z} Q_2(z) dz  dk.
\label{ag6}
}
Our goal next is to show that \eqref{a-eq-kdv} has a solution by proving that the map $a \mapsto \Phi_g[a]$ is a contraction in Sobolev spaces $H^m(0, T)$, $m\geq 0$.

\begin{remark}
The approach that led to the integral equation \eqref{a-int-eq} for the heat equation is not effective in the case of the linearized KdV equation. This is because, in the case of the heat equation, the reduction of the finite interval problem~\eqref{q-fi-ibvp} was accomplished through the half-line problems \eqref{vw-hl-ibvp-i}, which only involve the heat equation. As a consequence, the associated solution formulae \eqref{heat-utm-t} and \eqref{w-utm-t} only involve the complex contour $\p \mathcal D$ of Figure~\ref{d-heat-f}, along which the associated spatial exponentials decay.  
In the case of the positive-sign linearized KdV equation, however, the reduction of the finite interval problem \eqref{fi-kdv-red} includes the half-line problem \eqref{hlr-kdv-trans} for  the \textbf{negative-sign} linearized KdV equation. As the corresponding solution formula obtained through \eqref{w-kdv-sol} involves the complex contour $\mathcal C_-$ of Figure \ref{d-kdv-f}, with infinite portions along the real $k$-axis along which the relevant spatial exponential is \textbf{purely oscillatory} instead of decaying, a different approach is instead needed, leading to the integral equation~\eqref{a-eq-kdv}.
It is worth noting that the heat equation could also be analyzed via the approach used below for the linearized KdV equation.
\end{remark}

\subsection{Existence in $L^2(0, T)$}

We begin by considering \eqref{a-eq-kdv} in $H^0(0, T) \equiv L^2(0, T)$. We estimate the double integrals in \eqref{ag1} and  \eqref{ag4} collectively by considering 
\eee{\label{ag14}
P_j(t) := - \frac{1}{(2\pi)^2} \int_{k \in (-\infty, -1) \cup (1, \infty)} e^{ik \ell - i (k^3-k) t} 
 \left(k^j-\nu^j\right) \int_{z=0}^T e^{i (k^3-k) z}  Q_j(z) dz dk, \quad j = 1,2.
}
Making the change of variable $\tau = -\left(k^3-k\right)$, which maps $(-\infty, -1)$ to $(0, \infty)$ and $(1, \infty)$ to $(-\infty, 0)$, we obtain
\aaa{
P_j(t)
&=
- \frac{1}{(2\pi)^2} \int_{\infty}^{-\infty} e^{ik(\tau) \ell + i \tau t} 
 \left[k(\tau)^j-\nu(k(\tau))^j\right] \int_{z=0}^T e^{-i \tau z}  Q_j(z) dz \, \frac{d\tau}{1-3k(\tau)^2}
\nn\\
&=
\frac{1}{(2\pi)^2} \int_{\mathbb R} e^{i \tau t} \cdot e^{ik(\tau) \ell} 
\, \frac{k(\tau)^j-\nu(k(\tau))^j}{1-3k(\tau)^2} \int_{z=0}^T e^{-i \tau z}  Q_j(z) dz \, d\tau
}
where $k(\tau)$ satisfies the cubic equation $\tau = -\left(k^3-k\right)$ and is real when $\tau \in \mathbb R$  (since for $\tau \in \mathbb R$ the cubic equation has real coefficients, it always possesses at least one real root).
Then, noting that \eqref{ag14} makes sense for all $t\in\mathbb R$ due to the fact that the $t$-exponential is purely oscillatory, and subsequently using Plancherel's theorem, we have
\eee{
\no{P_j}_{L^2(0, T)}^2
\leq
\no{P_j}_{L_t^2(\mathbb R)}^2
=
\frac{1}{2\pi} \no{\mathcal F_t\left\{P_j\right\}}_{L_\tau^2(\mathbb R)}^2
=
\frac{1}{2\pi} 
\no{
\frac{1}{2\pi} e^{ik(\tau) \ell} 
\, \frac{k(\tau)^j-\nu(k(\tau))^j}{1-3k(\tau)^2} \int_{z=0}^T e^{-i \tau z}  Q_j(z) dz}_{L_\tau^2(\mathbb R)}^2.
\label{pjl2-0}
}
Since $\left|e^{ik(\tau)\ell}\right|=1$ for $\tau \in \mathbb R$ due to the  choice of $k(\tau)$, the above simplifies to
\eee{\label{p334}
\no{P_j}_{L^2(0, T)}^2
\leq
\frac{1}{(2\pi)^2} 
\int_{\tau\in\mathbb R}
\left|\frac{k(\tau)^j-\nu(k(\tau))^j}{1-3k(\tau)^2} \int_{z=0}^T e^{-i \tau z}  Q_j(z) dz\right|^2 d\tau.
}
Now, for $|k|\geq 1$, we have $\left|1-3k^2\right| \geq 3|k|^2 - 1 \geq 2|k|^2$. In addition,  recalling the definition of $\nu$ in \eqref{f-til-kdv}, 
\eee{\label{n-k3}
\left|k-\nu(k)\right|
\leq
|k| + |\nu(k)|
\leq
\frac{3|k|}{2} + \sqrt{1 + \frac{3|k|^2}{4}}
\leq
1 + 3|k|
}
and, consequently,
\eee{
\left|k^2-\nu(k)^2\right| 
\leq \left(|k|+|\nu(k)|\right)^2
\leq
\left(1+3|k|\right)^2.
\label{n-k4}
}
Hence, we have the bounds
\ddd{\label{pjl2-1}
&\left|\frac{k(\tau)-\nu(k(\tau))}{1-3k(\tau)^2}\right|
\leq
\frac{1+3|k|}{2|k|^2} \leq 2,  
\quad
\left|\frac{k(\tau)^2-\nu(k(\tau))^2}{1-3k(\tau)^2}\right|
\leq
\frac{\left(1+3|k|\right)^2}{2|k|^2} \leq 8, \quad |k|\geq 1,
}
where the ultimate inequality in each bound holds precisely because $|k|\geq 1$. In turn, \eqref{p334} yields
\eee{
\no{P_j}_{L^2(0, T)}^2
\leq
\frac{16}{\pi^2} 
\int_{\tau\in\mathbb R}
\left|\int_{z=0}^T e^{-i \tau z}  Q_j(z) dz\right|^2 d\tau
=
\frac{16}{\pi^2} \no{\mathcal F_z\left\{\chi_{[0, T]} Q_j\right\}}_{L_\tau^2(\mathbb R)}^2,
}
where the last equality follows from the fact that $Q_j(z)$ makes sense for all $z\in\mathbb R$ since $\text{Im}(\lambda^3-\lambda) = 0$ for $\lambda \in \mathcal C_+$. Thus, employing Plancherel's theorem once again and then using the definition \eqref{qj-kdv-def} of $Q_j$, we find
\aaa{
\no{P_j}_{L^2(0, T)}^2
&\leq
\frac{32}{\pi} \no{\chi_{[0, T]} Q_j}_{L_z^2(0, T)}^2
=
\frac{32}{\pi} \int_0^T \left|\int_{\lambda \in \mathcal C_+} e^{i\lambda\ell + i (\lambda^3-\lambda) z} \lambda^{2-j}  \left(1-3\lambda^2\right) \int_{r=0}^T e^{-i(\lambda^3-\lambda)r} a(r) dr d\lambda\right|^2 dz.
\nn
}
Introducing the notation 
\eee{\label{j12-kdv-def}
J_j(\sigma) := \int_{\lambda \in \mathcal C_+} e^{i\lambda\ell + i (\lambda^3-\lambda) \sigma} \lambda^{2-j} \left(1-3\lambda^2\right) d\lambda,
\quad j=1, 2,
}
and using the triangle inequality, we further have
\eee{\label{p336}
\no{P_j}_{L^2(0, T)}^2
\leq
\frac{32}{\pi} \int_0^T \left|
\int_{r=0}^T J_j(z-r)  a(r) dr \right|^2 dz
\leq
\frac{32}{\pi} \int_0^T \left(
\int_{r=0}^T \left|J_j(z-r)\right|  \left|a(r)\right| dr \right)^2 dz.
}

Let us now estimate $J_j(\sigma)$. Parametrizing along $\mathcal C_+$ according to \eqref{clr-kdv} and taking into account the contour's orientation, we write
\aaa{
J_1(\sigma) 
&= 
\int_\infty^0 e^{-i\overline{\gamma(y)} \ell - i \overline{(\gamma(y)^3 - \gamma(y))} \sigma} \, \overline{\gamma(y)} \, \left(1-3 \overline{\gamma(y)^2}\right) \overline{\gamma'(y)} dy
+
 \int_{-\frac{1}{\sqrt 3}}^{\frac{1}{\sqrt 3}} e^{iy\ell + i (y^3-y) \sigma} y \left(1-3y^2\right) dy
\nn\\
&\quad
+
\int_0^\infty e^{i \gamma(y) \ell + i ((\gamma(y))^3-\gamma(y)) \sigma} \gamma(y) \left(1-3\gamma(y)^2\right) \gamma'(y) dy.
\nn
}
Hence, noting that $\text{Im}(\gamma(y)^3-\gamma(y)) = 0$ for $y\geq 0$ (since $\text{Im}(k^3-k) = 0$ along $\mathcal C_R \cup \mathcal C_L$) and $\gamma'(y) = \frac{y}{\sqrt{3(y^2+1)}} + i$, 
\aaa{
\left| J_1(\sigma) \right|
&\leq
2 \int_0^\infty e^{-y \ell} \sqrt{\frac{4y^2+1}{3}} \left(1+3 \cdot \frac{4y^2+1}{3}\right) \sqrt{\frac{y^2}{3(y^2+1)} + 1} \, dy
+
2\int_{0}^{\frac{1}{\sqrt 3}} y \left(1-3y^2\right) dy
\nn\\
&\leq 
\frac{4}{\sqrt 3}  \int_0^\infty e^{-y \ell} \sqrt{4y^2+1} \left(2y^2+1\right) \sqrt{\frac 13 + 1}  \, dy
+
2\int_{0}^{\frac{1}{\sqrt 3}} y \left(1-3y^2\right) dy
\nn\\
&\leq
\frac{8}{3}  \int_0^\infty e^{-y \ell} \left(2y+1\right) \left(2y^2+1\right)  dy
+
2\int_{0}^{\frac{1}{\sqrt 3}} y \left(1-3y^2\right) dy
=
\frac{8\left(\ell^3 + 2\ell^2 + 4\ell + 24\right)}{3\ell^4}
+
\frac 16,
\label{j1-bound-kdv}
}
which is a uniform bound for $J_1$ with respect to $\sigma \in \mathbb R$. Analogously, $J_2$ is shown to admit the uniform bound
\eee{\label{j2-bound-kdv}
\left| J_2(\sigma) \right|
\leq
\frac{8\left(\ell^2+4\right)}{\sqrt 3 \, \ell^3} + \frac{4}{3\sqrt 3}.
}

Combining \eqref{p336} with the bounds \eqref{j1-bound-kdv} and \eqref{j2-bound-kdv}, we obtain
\eee{
\no{P_j}_{L^2(0, T)}
\leq
\frac{4 \sqrt 2}{\sqrt \pi} \, c_j 
\left(\int_0^T \left(
\int_{r=0}^T  \left|a(r)\right| dr \right)^2 dz\right)^{\frac 12}
=
\frac{4 \sqrt 2}{\sqrt \pi} \, c_j 
\sqrt T \no{a}_{L^1(0, T)}
\leq
\frac{4 \sqrt 2}{\sqrt \pi} \, c_j T \no{a}_{L^2(0, T)}
\label{ag14-bound}
}
where
\eee{\label{cj-kdv-def}
c_1 = \frac{8\left(\ell^3 + 2\ell^2 + 4\ell + 24\right)}{3\ell^4} + \frac 16,
\quad
c_2 = \frac{8\left(\ell^2+4\right)}{\sqrt 3 \, \ell^3} + \frac{4}{3\sqrt 3}.
}

We proceed to the terms \eqref{ag2} and \eqref{ag5}, which will be estimated collectively by considering 
\eee{\label{ag25}
R_j(t) := \frac{1}{(2\pi)^2} \int_{k \in (-1, -\frac{1}{\sqrt 3}) \cup (\frac{1}{\sqrt 3}, 1)} e^{ik \ell - i (k^3-k) t} 
 \left(k^j-\nu^j\right) \int_{z=0}^T e^{i (k^3-k) z}  Q_j(z) dz dk, \quad j = 1,2.
}
As the range of integration is finite, we employ \eqref{qj-kdv-def}, \eqref{n-k3}, \eqref{n-k4} and \eqref{j12-kdv-def} to infer 
\aaa{
\no{R_j}_{L^2(0, T)}^2
&\leq
\frac{1}{(2\pi)^4} 
\int_{t=0}^T \left(\int_{k \in (-1, -\frac{1}{\sqrt 3}) \cup (\frac{1}{\sqrt 3}, 1)}  
 \left|k^j-\nu^j\right| \int_{z=0}^T \left|Q_j(z)\right| dz dk\right)^2 dt
\nn\\
&\leq
\frac{1}{(2\pi)^4} 
\int_{t=0}^T \Bigg(\int_{k \in (-1, -\frac{1}{\sqrt 3}) \cup (\frac{1}{\sqrt 3}, 1)}  
\left(1+3|k|\right)^j \int_{z=0}^T \Bigg|
\int_{\lambda \in \mathcal C_+} e^{i\lambda\ell + i (\lambda^3-\lambda) z} \lambda^{2-j} \left(1-3\lambda^2\right) 
\nn\\
&\hskip 8.5cm \cdot \int_{r=0}^T e^{-i(\lambda^3-\lambda)r} a(r) dr d\lambda
\Bigg| dz dk\Bigg)^2 dt
\nn\\
&\leq
\frac{1}{(2\pi)^4} 
\int_{t=0}^T \left(2 \int_{k \in (\frac{1}{\sqrt 3}, 1)} 4^j dk\right) ^2 
\left(
 \int_{z=0}^T \left|
 \int_{r=0}^T a(r) J_j(z-r) dr
\right| dz\right)^2 dt
\nn\\
&\leq
\frac{2^{4j+2} \left(1-\frac{1}{\sqrt 3}\right)^2}{(2\pi)^4} 
\int_{t=0}^T
\left(
 \int_{z=0}^T 
 \int_{r=0}^T  \left|J_j(z-r)\right| \left|a(r)\right| dr dz\right)^2 dt.
 }
Hence, in view of the uniform bounds \eqref{j1-bound-kdv} and \eqref{j2-bound-kdv} for $J_j(\sigma)$ with $\sigma \in \mathbb R$, 
\aaa{
\no{R_j}_{L^2(0, T)}
&\leq
\frac{2^{2j-1} \left(\sqrt 3-1\right)}{\pi^2 \sqrt 3} 
\left(
\int_{t=0}^T
\left(
 \int_{z=0}^T 
 \int_{r=0}^T  c_j \left|a(r)\right| dr dz\right)^2 dt
\right)^{\frac 12}
\nn\\
&=
\frac{2^{2j-1} \left(\sqrt 3-1\right) c_j}{\pi^2 \sqrt 3} \, 
T^{\frac 32}  \no{a}_{L^1(0, T)}
\leq
\frac{2^{2j-1} \left(\sqrt 3-1\right) c_j}{\pi^2 \sqrt 3} \, 
T^2  \no{a}_{L^2(0, T)}
\label{ag25-bound}
}
with the constants $c_j$ given by \eqref{cj-kdv-def}. 

Lastly, we estimate the terms \eqref{ag3} and \eqref{ag6} by considering 
\eee{\label{ag36}
S_j(t) := \frac{1}{(2\pi)^2} \int_{k\in \mathcal C_L\cup\mathcal C_R} e^{ik \ell - i (k^3-k) t} 
 \left(k^j-\nu^j\right) \int_{z=0}^T e^{i (k^3-k) z} Q_j(z) dz dk, \quad j=1,2. 
}
For this term, we exploit the fact that along the contours $\mathcal C_L$ and $\mathcal C_R$ we have exponential decay from $e^{ik\ell}$ while the time exponentials are unitary since $\text{Im}(k^3-k) = 0$. Specifically, substituting for $Q_j$ via \eqref{qj-kdv-def} and rearranging the order of integration, we are able to express $S_j(t)$ in the form 
\eee{
S_j(t) 
=
\frac{1}{(2\pi)^2} \int_{z=0}^T M_j(t-z)  \int_{r=0}^T J_j(z-r)  a(r) dr  dz 
}
where  $J_j$ is defined by  \eqref{j12-kdv-def} and 
\eee{\label{mj-kdv-def}
M_j(\sigma) 
:= 
\int_{k\in \mathcal C_L\cup\mathcal C_R} e^{ik \ell - i (k^3-k)\sigma} \left(k^j-\nu^j\right) dk, \quad j=1,2.
} 
Then, using the bounds \eqref{j1-bound-kdv} and \eqref{j2-bound-kdv}, we find
\aaa{
\no{S_j}_{L^2(0, T)}
&\leq
\frac{1}{(2\pi)^2} 
\no{\int_{z=0}^T \left|M_j(t-z)\right|  \int_{r=0}^T \left|J_j(z-r)\right|  \left|a(r)\right| dr  dz}_{L_t^2(0, T)}
\nn\\
&\leq
\frac{c_j}{(2\pi)^2} 
\no{\int_{z=0}^T \left|M_j(t-z)\right|  dz}_{L_t^2(0, T)} \sqrt T \no{a}_{L^2(0, T)}
\label{ag36-bound-0}
}
with $c_j$ given by \eqref{cj-kdv-def}. Thus, it suffices to bound $M_j(\sigma)$ for $\sigma \in \mathbb R$. 

Parametrizing along $\mathcal C_L \cup \mathcal C_R$ according to \eqref{clr-kdv} and taking into account  the contour orientation, we have 
\aaa{
M_j(\sigma) 
&=
- \int_0^\infty e^{-i\overline{\gamma(y)} \ell + i \overline{(\gamma(y)^3 - \gamma(y))} \sigma} \left(\big(-\overline{\gamma(y)}\big)^j - \nu(-\overline{\gamma(y)})^j\right)  \overline{\gamma'(y)} dy
\nn\\
&\quad
+
\int_\infty^0 e^{i \gamma(y) \ell - i ((\gamma(y))^3-\gamma(y)) \sigma}  
\left(\gamma(y)^j - \nu(\gamma(y))^j\right) \gamma'(y) dy.
}
Hence, noting that $\text{Im}(\gamma(y)^3-\gamma(y)) = 0$ for $y\geq 0$ (since $\text{Im}(k^3-k) = 0$ along $\mathcal C_R \cup \mathcal C_L$), we find
\aaa{
\left|M_j(\sigma)\right|
&\leq
\int_0^\infty e^{-\ell y} \left|\big(-\overline{\gamma(y)}\big)^j - \nu(-\overline{\gamma(y)})^j\right|  \left|\gamma'(y)\right| dy
+
\int_0^\infty e^{-\ell y} 
\left|\gamma(y)^j - \nu(\gamma(y))^j\right| \left|\gamma'(y)\right| dy
\nn\\
&\leq
\int_0^\infty e^{-\ell y} \left(\left|\gamma(y)\right|^j + \big|\nu(-\overline{\gamma(y)})\big|^j\right) \left|\gamma'(y)\right| dy
+
\int_0^\infty e^{-\ell y} 
\left(\left|\gamma(y)\right|^j + \left|\nu(\gamma(y))\right|^j\right) \left|\gamma'(y)\right| dy.
}
and then by \eqref{n-k3}, \eqref{n-k4} and the fact that $\gamma'(y) = \frac{y}{\sqrt{3(y^2+1)}} + i$ we obtain
\aaa{
\left|M_j(\sigma)\right|
&\leq
2 \int_0^\infty e^{-\ell y} \left(1+3\left|\gamma(y)\right|\right)^j  \sqrt{\frac{y^2}{3(y^2+1)} + 1} \, dy
\nn\\
&\leq
2 \int_0^\infty e^{-\ell y} \left(1+3 \sqrt{\frac{4y^2+1}{3}}\right)^j  \sqrt{\frac 13 + 1} \, dy
\nn\\
&\leq
\frac{4}{\sqrt 3} \int_0^\infty e^{-\ell y} \cdot 2\left(1+3 \left(4y^2+1\right)\right)  dy
=
\frac{2^5 \left(\ell^2+6\right)}{\sqrt 3 \, \ell^3}.  
}
In turn, \eqref{ag36-bound-0} yields 
\eee{\label{ag36-bound}
\no{S_j}_{L^2(0, T)}
\leq
\frac{c_j}{(2\pi)^2} 
\no{\int_{z=0}^T\frac{2^5 \left(\ell^2+6\right)}{\sqrt 3 \, \ell^3} dz}_{L_t^2(0, T)} \sqrt T \no{a}_{L^2(0, T)}
=
\frac{2^3 c_j \left(\ell^2+6\right)}{\pi^2 \sqrt 3 \, \ell^3} \, T^2 \no{a}_{L^2(0, T)}.
}

Overall, combining \eqref{ag14-bound}, \eqref{ag25-bound}, \eqref{ag36-bound} with the definition \eqref{ag1} of $\Phi_g[a]$, we infer the estimate
\eee{
\no{\Phi_g[a]}_{L^2(0, T)}
\leq
\no{g}_{L^2(0, T)}
+
\frac{2^{\frac 52} \left(c_1 + c_2\right)}{\sqrt \pi} 
\left(  
1 
+
\frac{\sqrt 2 \left(\sqrt 3-1\right)}{\pi^{\frac 32} \sqrt 3} 
T
+
\frac{\sqrt 2 \left(\ell^2+6\right)}{\pi^{\frac 32} \sqrt 3 \, \ell^3}   T
\right) T \no{a}_{L^2(0, T)}
\label{phiga-l2}
}
with $c_1, c_2$ given by \eqref{cj-kdv-def}.
Moreover, thanks to the linearity of $\Phi_g[a]$ in $a$, a similar estimation yields 
\eee{
\no{\Phi_g[a_1] - \Phi_g[a_2]}_{L^2(0, T)}
\leq
\frac{2^{\frac 52} \left(c_1 + c_2\right)}{\sqrt \pi} 
\left(  
1 
+
\frac{\sqrt 2 \left(\sqrt 3-1\right)}{\pi^{\frac 32} \sqrt 3} 
T
+
\frac{\sqrt 2 \left(\ell^2+6\right)}{\pi^{\frac 32} \sqrt 3 \, \ell^3}   T
\right) T \no{a_1 - a_2}_{L^2(0, T)}
\label{phiga-l2-contr}
}
for any $a_1, a_2 \in L^2(0, T)$. 

With the two above estimates at hand, letting $B(0, \rho) \subset L^2(0, T)$ denote the closed ball of radius $\rho = 2\no{g}_{L^2(0, T)}$ centered at zero, for any $a \in B(0, \rho)$ we have
$$
\no{\Phi_g[a]}_{L^2(0, T)}  
\leq 
\frac \rho 2 + \frac{2^{\frac 52} \left(c_1 + c_2\right)}{\sqrt \pi} 
\left(  
1 
+
\frac{\sqrt 2 \left(\sqrt 3-1\right)}{\pi^{\frac 32} \sqrt 3} 
T
+
\frac{\sqrt 2 \left(\ell^2+6\right)}{\pi^{\frac 32} \sqrt 3 \, \ell^3}   T
\right) T
\rho 
$$
so if $T>0$ is such that, for $c_1, c_2$ given by \eqref{cj-kdv-def}, 
\eee{\label{T-L2-kdv}
\frac{2^{\frac 52} \left(c_1 + c_2\right)}{\sqrt \pi} 
\left(  
1 
+
\frac{\sqrt 2 \left(\sqrt 3-1\right)}{\pi^{\frac 32} \sqrt 3} 
T
+
\frac{\sqrt 2 \left(\ell^2+6\right)}{\pi^{\frac 32} \sqrt 3 \, \ell^3}   T
\right) T
\leq \frac 12
}
then $\Phi_g[a] \in B(0, \rho)$. 
Furthermore, for such a $T>0$, estimate \eqref{phiga-l2-contr} implies that the map $a \mapsto \Phi_g[a]$ is a contraction on $B(0, \rho)$.
Thus, by Banach's fixed point theorem, $\Phi_g[a]$ has a unique fixed point in $B(0, \rho)$, which amounts to a unique solution of the integral equation~\eqref{a-eq-kdv} for $a$ in  $B(0, \rho)$.  
Furthermore, having proved the existence of such a solution as a fixed point of $\Phi_g[a]$, we can return to \eqref{phiga-l2} and obtain the improved size estimate 
\begin{equation}\label{a-l2-est-kdv}
\no{a}_{L^2(0, T)} 
\leq
\frac{1}{1-\frac{2^{\frac 52} \left(c_1 + c_2\right)}{\sqrt \pi} 
\left(  
1 
+
\frac{\sqrt 2 \left(\sqrt 3-1\right)}{\pi^{\frac 32} \sqrt 3} 
T
+
\frac{\sqrt 2 \left(\ell^2+6\right)}{\pi^{\frac 32} \sqrt 3 \, \ell^3}   T
\right) T} \no{g}_{L^2(0, T)}.
\end{equation}
In summary, if $T>0$ satisfies \eqref{T-L2-kdv}, then there exists a unique $a \in B(0, 2\no{g}_{L^2(0, T)})\subset L^2(0, T)$ that solves the integral equation \eqref{a-eq-kdv}.

\subsection{Existence in $H^m(0, T)$ for any $m\geq 0$}

Let us denote the solutions to problems \eqref{vw-hllr-i}  by $v[a]$ and $w[b, c]$. Differentiating both of these problems with respect to $t$ and setting $v_1[a] := \p_t \left(v[a]\right)$ and $w_1[b, c] := \p_t \left(w[b, c]\right)$, we have
\begin{equation}\label{p-kdv}
\begin{aligned}
&\p_t \left(v_1[a]\right) + \p_x \left(v_1[a]\right) + \p_x^3 \left(v_1[a]\right)= 0, \quad x \in (0, \infty), \ t \in (0, T),
\\
&v_1[a](x, 0) = 0, \quad x \in (0, \infty),
\\
&v_1[a](0, t) = a'(t), \quad t \in (0, T),
\end{aligned}
\end{equation}
and  
\begin{equation}\label{q-kdv}
\begin{aligned}
&\p_t \left(w_1[b, c]\right) + \p_x \left(w_1[b, c]\right) + \p_x^3 \left(w_1[b, c]\right)= 0, \quad x \in (-\infty, \ell), \ t \in (0, T),
\\
&w_1[b, c](x, 0) = 0, \quad x \in (-\infty, \ell),
\\
&w_1[b, c](\ell, t) = b'(t), \quad \p_x \left(w_1[b, c]\right)(\ell, t) = c'(t), \quad t \in (0, T).
\end{aligned}
\end{equation}
Note that the initial datum in problem \eqref{p-kdv}  is zero due to the fact that, by the first of the problems in \eqref{vw-hllr-i}, $v_1[a] \equiv \p_t \left(v[a]\right) = -\p_x \left(v[a]\right) -\p_x^3 \left(v[a]\right)$ and the right side vanishes at $t=0$ since $v[a](x, 0) = 0$. Through the same reasoning, problem problem \eqref{q-kdv} implies that $w_1[b, c](x, 0) = 0$. 

Combining \eqref{p-kdv} and \eqref{q-kdv} with induction, it follows that for any $n\in\mathbb N_0$ the derivatives
\eee{\label{vnwn-def}
v_n[a](x, t) := \p_t^n \left(v[a](x, t)\right), \quad w_n[b, c](x, t) := \p_t^n \left(w[b, c](x, t)\right)
}
satisfy the problems
\begin{equation}\label{vn-kdv}
\begin{aligned}
&\p_t \left(v_n[a]\right) + \p_x \left(v_n[a]\right) + \p_x^3 \left(v_n[a]\right) = 0, \quad x \in (0, \infty), \ t \in (0, T),
\\
&v_n[a](x, 0) = 0, \quad x \in (0, \infty),
\\
&v_n[a](0, t) = a^{(n)}(t), \quad t \in (0, T),
\end{aligned}
\end{equation}
and  
\begin{equation}\label{wn-kdv}
\begin{aligned}
&\p_t \left(w_n[b, c]\right) + \p_x \left(w_n[b, c]\right) + \p_x^3 \left(w_n[b, c]\right) = 0, \quad x \in (-\infty, \ell), \ t \in (0, T),
\\
&w_n[b, c](x, 0) = 0, \quad x \in (-\infty, \ell),
\\
&w_n[b, c](\ell, t) = b^{(n)}(t), \quad \p_x \left(w_n[b, c]\right)(\ell, t) = c^{(n)}(t), \quad t \in (0, T).
\end{aligned}
\end{equation}
The problems \eqref{vn-kdv} and \eqref{wn-kdv} are identical to the two problems in \eqref{vw-hllr-i} except for the fact that the various boundary data are replaced by their $n$th derivatives. 
Therefore, by uniqueness of solution to these two linear problems,
\eee{\label{vwn}
v_n[a](x, t) = v[a^{(n)}](x, t), \quad
w_n[b, c](x, t) = w[b^{(n)}, c^{(n)}](x, t),
\quad
n \in \mathbb N_0.
}

Furthermore, differentiating the superposition \eqref{sup-kdv-i} $n$ times in $t$ we obtain
\begin{equation}\label{sup-kdv-n}
q_n[g](x, t) = v_n[a](x, t)\big|_{x\in(0, \ell)} + w_n[b, c](x, t)\big|_{x\in(0, \ell)}
\end{equation}
where, via the same reasoning as above, the function $q_n[g](x, t) := \p_t^n\left(q[g](x, t)\right)$ satisfies the problem 
\begin{equation}\label{fi-kdv-red-n}
\begin{aligned}
&\p_t \left(q_n[g]\right) + \p_x \left(q_n[g]\right) + \p_x^3 \left(q_n[g]\right) = 0, \quad x \in (0, \ell), \ t \in (0, T),
\\
&q_n[g](x, 0) = 0, \quad x \in (0, \ell),
\\
&q_n[g](0, t) = g^{(n)}(t), \quad q_n[g](\ell, t) = 0, \quad \p_x \left(q_n[g]\right)(\ell, t) = 0, \quad t \in (0, T),
\end{aligned}
\end{equation}
and hence
\eee{\label{qn}
q_n[g](x, t) = q[g^{(n)}](x, t), 
\quad
n \in \mathbb N_0.
} 
Combining \eqref{vwn} and \eqref{qn} with \eqref{abc-cond} (or, equivalently, proceeding directly via \eqref{sup-kdv-n}), we obtain the conditions
\ddd{\label{abc-cond-n}
a^{(n)}(t) &= g^{(n)}(t) - w[b^{(n)}, c^{(n)}](0, t),
\\
b^{(n)}(t) &= -v[a^{(n)}](\ell, t),
\quad
c^{(n)}(t) = -v_x[a^{(n)}](\ell, t),
}
which can be combined into the integral equation
\eee{\label{ap-eq-kdv}
a^{(n)}(t) 
=
g^{(n)}(t) + w\big[v[a^{(n)}]|_{x=\ell}, v_x[a^{(n)}]|_{x=\ell}\big](0, t).
}

Noting that \eqref{ag1}-\eqref{ag6} together with the formulae \eqref{v-kdv-sol} and \eqref{w-kdv-sol} imply
\eee{\label{phin-def}
\Phi_g[a](t) = g(t) + w\big[v[a]|_{x=\ell}, v_x[a]|_{x=\ell}\big](0, t),
}
the integral equation \eqref{ap-eq-kdv} reads
\eee{
a^{(n)}(t) = \Phi_{g^{(n)}}[a^{(n)}](t), \quad n\in\mathbb N_0,
}
which is the integral equation \eqref{a-eq-kdv} with $a^{(n)}$ and $g^{(n)}$ in place of $a$ and $g$, respectively.
Moreover, taking the $n$th (time) derivative of \eqref{phin-def} and then recalling the notation \eqref{vnwn-def} and the equalities \eqref{vwn}, we have
\aaa{
\Phi_g[a]^{(n)}(t) 
&=
g^{(n)}(t) + w_n\big[v[a]|_{x=\ell}, v_x[a]|_{x=\ell}\big](0, t)
=
g^{(n)}(t) + w\big[\p_t^n \left(v[a]|_{x=\ell}\right), \p_t^n \left(v_x[a]|_{x=\ell}\right)\big](0, t)
\nn\\
&=
g^{(n)}(t) + w\big[v_n[a]|_{x=\ell},  \p_t^n \left(\left(\p_x v[a](x, t)\right) |_{x=\ell}\right)\big](0, t)
=
g^{(n)}(t) + w\big[v_n[a]|_{x=\ell},  \left(\p_x v_n[a](x, t) \right)|_{x=\ell}\big](0, t)
\nn\\
&=
g^{(n)}(t) + w\big[v[a^{(n)}]|_{x=\ell}, v_x[a^{(n)}]|_{x=\ell}\big](0, t)
=
\Phi_{g^{(n)}}[a^{(n)}](t).
\label{phinn}
}

Combining \eqref{phinn} with the definition \eqref{sob-int} of the Sobolev norm, we deduce
\eee{
\no{\Phi_g[a]}_{H^m(0, T)}
=
\sum_{n=0}^m \big\|\Phi_{g^{(n)}}[a^{(n)}]\big\|_{L^2(0, T)},
\quad
m \in \mathbb N_0.
}
Furthermore, by estimate \eqref{phiga-l2}, for each $n\in\mathbb N_0$ we have
\aaa{
\big\|\Phi_{g^{(n)}}[a^{(n)}]\big\|_{L^2(0, T)}
&\leq
\big\|g^{(n)}\big\|_{L^2(0, T)}
+
\frac{2^{\frac 52} \left(c_1 + c_2\right)}{\sqrt \pi} 
\left(  
1 
+
\frac{\sqrt 2 \left(\sqrt 3-1\right)}{\pi^{\frac 32} \sqrt 3} 
T
+
\frac{\sqrt 2 \left(\ell^2+6\right)}{\pi^{\frac 32} \sqrt 3 \, \ell^3}   T
\right) T \, \big\|a^{(n)}\big\|_{L^2(0, T)}
\label{phiga-l2-n}
}
with $c_1, c_2$ given by \eqref{cj-kdv-def}. Thus, we conclude that
\eee{\label{phiga-hm-int}
\no{\Phi_g[a]}_{H^m(0, T)}
\leq
\no{g}_{H^m(0, T)}
+
\frac{2^{\frac 52} \left(c_1 + c_2\right)}{\sqrt \pi} 
\left(  
1 
+
\frac{\sqrt 2 \left(\sqrt 3-1\right)}{\pi^{\frac 32} \sqrt 3} 
T
+
\frac{\sqrt 2 \left(\ell^2+6\right)}{\pi^{\frac 32} \sqrt 3 \, \ell^3}   T
\right) T \no{a}_{H^m(0, T)}
}
for any $m\in\mathbb N_0$. 
Moreover, the interpolation result of \eqref{inter-t} can be used to extend the validity of estimate~\eqref{phiga-hm-int} to any $m\geq 0$. 

Estimate \eqref{phiga-hm-int} is entirely analogous to the $L^2$-estimate \eqref{phiga-l2}. Furthermore, adjusting its derivation accordingly, we can easily obtain the analogue of the contraction inequality \eqref{phiga-l2-contr}. Thus, for any $m\geq 0$ and $T>0$ satisfying~\eqref{T-L2-kdv}, there is a unique $a \in B(0, \rho) \subset H^m(0, T)$ with $\rho = 2\no{g}_{H^m(0, T)}$ that satisfies the integral equation \eqref{a-eq-kdv} and admits the size estimate
\begin{equation}\label{a-hm-est-kdv}
\no{a}_{H^m(0, T)} 
\leq
\frac{1}{1-\frac{2^{\frac 52} \left(c_1 + c_2\right)}{\sqrt \pi} 
\left(  
1 
+
\frac{\sqrt 2 \left(\sqrt 3-1\right)}{\pi^{\frac 32} \sqrt 3} 
T
+
\frac{\sqrt 2 \left(\ell^2+6\right)}{\pi^{\frac 32} \sqrt 3 \, \ell^3}   T
\right) T}
\no{g}_{H^m(0, T)}, \quad m \geq 0.
\end{equation}

\section{Concluding remarks}
\label{c-s}

The techniques introduced in Sections \ref{red-s} and \ref{kdv-s} are general, in the sense that they can be applied to a wide range of linear evolution equations of dispersive or dissipative nature. By extension, the well-posedness of nonlinear counterparts of these equations on a finite interval can be inferred from the relevant well-posedness results on the half-line.
A notable exception is the nonlinear Schr\"odinger equation, whose well-posedness theory on a finite interval is known to require additional smoothness of the associated boundary data compared to the one required in the case of the half-line \cite{bsz2018}. 
In particular, neither the heat equation method of Section \ref{red-s} nor the linearized KdV method of Section \ref{kdv-s} are readily applicable to the linear Schr\"odinger equation, due to the fact that the complex contour of integration involved in the relevant unified transform solution formula (i.e. the analogue of the contours $\p \mathcal D$ and $\mathcal C^+$ in \eqref{heat-utm-t} and \eqref{v-kdv-sol}) is the boundary of the first quadrant of the complex plane, thus involving an \textit{infinite} portion along which the relevant exponential $e^{ikx-ik^2 t}$ is purely oscillatory (see, for example, formula (1.16) in \cite{fhm2017}).
The adaptation of the method introduced in the present work to the framework of Schr\"odinger-type equations is an interesting open problem that will be the subject of a future work.

\subsection{Uniqueness and global solvability}
\label{ug-ss}

The integral equation \eqref{a-int-eq} can have at most one solution on any interval $[0, T]$ with arbitrary $T>0$. 
The proof of this uniqueness result can be reduced to the case of the homogeneous problem with $g(t) \equiv 0$. It then suffices to show that $a(t) \equiv 0$. If $T<T_0$ with $T_0$ satisfying the equality in the contraction condition \eqref{T-L2}, then we use inequality \eqref{a-l2-est} to infer that $a(t) \equiv 0$ in $L^2(0, T)$. If $T>T_0$, then we first solve the homogeneous version of the integral equation \eqref{a-int-eq} on $[0, T_0]$ in order to infer that $a(t)\equiv 0$ on that interval. Subsequently, exploiting the fact that the contraction condition \eqref{T-L2} only depends on $\ell$, we solve the integral equation on the interval $[T_0/2, 3 T_0/2]$ in order to conclude that $a(t)\equiv 0$ also on that latter interval. This process can be repeated until the desired interval $[0, T]$ is covered. 

Having proved uniqueness, we next show that the integral equation \eqref{a-int-eq}, which was solved in Section \ref{red-s} up to times satisfying the contraction condition~\eqref{T-L2}, can actually be solved globally in the sense that its original solution can be extended to arbitrary time. Indeed, suppose that the solution to \eqref{a-int-eq} is originally obtained via contraction on the interval $[0, T_0]$ with $T_0$ satisfying the equality in~\eqref{T-L2}. 
Then, thanks to the fact that \eqref{T-L2} is \textit{independent} of the boundary datum $g$ (whose norm is used to define the radius of the ball for the contraction), we  carry out a contraction argument to obtain a solution to \eqref{a-int-eq} on the interval $\left[T_0/2, 3T_0/2\right]$, whose length is the same with that of the original interval $[0, T_0]$.  
Furthermore, by uniqueness (see  above), the solution on $[T_0/2, 3T_0/2]$ and the solution  on $[0, T_0]$  are equal on the overlapping interval $[T_0/2, T_0]$, hence eliminating any regularity concerns due to the ``gluing'' of the solutions. This process can be repeated until the desired time of existence is reached.

\subsection{Numerical illustrations}

We now provide a numerical illustration of the key components of the present work in the case of the heat equation. First, in Figures \ref{fig:3D-ab-s} and \ref{fig:3D-ab-s-1}, we observe the validity of the decomposition \eqref{fi-hl-dec-i} of the finite interval problem \eqref{q-fi-ibvp-i} into the two half-line problems \eqref{vw-hl-ibvp-i}. The main ingredient here is the numerical solution of the integral equation \eqref{a-int-eq}, which provides the half-line boundary datum $a(t)$ from the knowledge of the finite interval boundary datum $g(t)$. Then, the equation \eqref{b-sys} yields the boundary datum $b(t)$, which is the second piece of data involved in the decomposition \eqref{fi-hl-dec-i}. 

The above illustration is performed via the following simple numerical scheme. First, for sufficiently large $\tau>0$ and integers $M, N>0$, we write
\eee{\label{ab-num}
a(t)=
\left\{\begin{array}{ll}
\displaystyle \sum _{n=1}^M c_n \sin \left(\frac{ \pi  n t}{\tau}\right), & 0<t<\tau, 
\\[4mm]
 0, & t>\tau,
\end{array}
\right.
\qquad
  b(t)=
\left\{\begin{array}{ll}
\displaystyle \sum _{n=1}^N f_n \sin \left(\frac{ \pi  n t}{\tau}\right), & 0<t<\tau, \\[4mm]
 0, & t>\tau,
 \end{array}
\right.
}
thereby imposing finite support for the boundary conditions, as well as the compatibility condition $a(0)=b(0)=0$. By numerically evaluating the integral equation \eqref{a-int-eq} at $\left\{t_k \right\}_{k=1}^M$, we produce a linear algebraic system $M\times M$ for~$\left\{c_n \right\}_{n=1}^M$. Then, \eqref{b-sys} provides the numerical value of $b(t)$. Similarly, we determine $\left\{f_n \right\}_{n=1}^N$ via a $N\times N$  linear algebraic system. 

\begin{remark}
The choice of basis in \eqref{ab-num} is,  of course, restrictive in the reconstruction of $a(t)$. For example, the presence of a discontinuous boundary datum $g(t)$ would require a suitable basis for the representation of $a(t)$ incorporating that feature. A thorough numerical reconstruction of $a(t)$ via \eqref{a-int-eq} is outside the scope of the present work; rather, our aim is to merely provide a numerical  illustration of the associated theoretical analysis. 
\end{remark}

The evaluations displayed in Figure~\ref{fig:3D-ab-s} were performed for  small enough $T>0$ so that the contraction condition~\eqref{T-L2} is satisfied, namely for $T\approx 6 < 4^2 = \ell^2$, which yields a half-line datum $a(t)$ ``quite close'' to the finite interval datum $g(t)$, as expected by the contraction mapping. The solutions $v(x,t)$ and $w(x,t)$ of the half-line problems~\eqref{vw-hl-ibvp-i} are plotted using the unified transform formulae \eqref{heat-utm-t} and \eqref{w-utm-t}, respectively. The restriction of these two functions on $x\in(0,\ell)$, as well as their sum $v+w$, are illustrated in the top panel of Figure~\ref{fig:3D-ab-s}. The latter surface, shown in green, coincides (indistinguishable  in Figure~\ref{fig:3D-ab-s}) with the solution $q(x,t)$ to the finite interval problem \eqref{q-fi-ibvp-i},  which is computed via the unified transform formula (see (2.6), (2.11) in \cite{f2008})
\begin{equation}\label{q-sol}
    q(x,t)=\frac{1}{i\pi} \int_{k\in\partial \mathcal D} \frac{e^{-k^2 t}}{\sin(k\ell)} \sin[k(\ell-x)] \, k \, \widetilde{g}(k^2,t) \, dk
\end{equation}
with $\mathcal D$ and $\widetilde{g}(k^2,t)$ defined by \eqref{d-def} and \eqref{atil-def}, respectively. 
This result numerically verifies the decomposition~\eqref{fi-hl-dec-i}, as also shown in the bottom panels of Figure \ref{fig:3D-ab-s}, where an $\mathcal O\left(10^{-5}\right)$ discrepancy  between $q$ and $v+w$ is displayed.

\begin{figure}[ht!]
    \centering
    \includegraphics[scale=0.7]{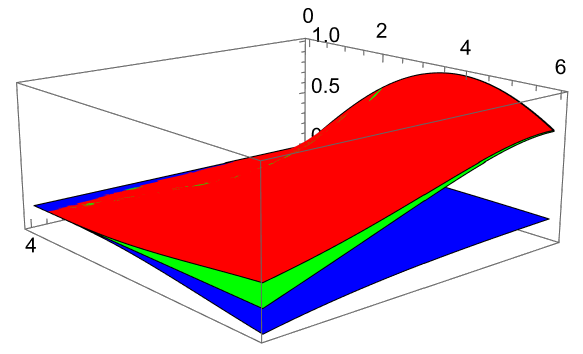} 
    \\
     \includegraphics[width=0.35\linewidth]{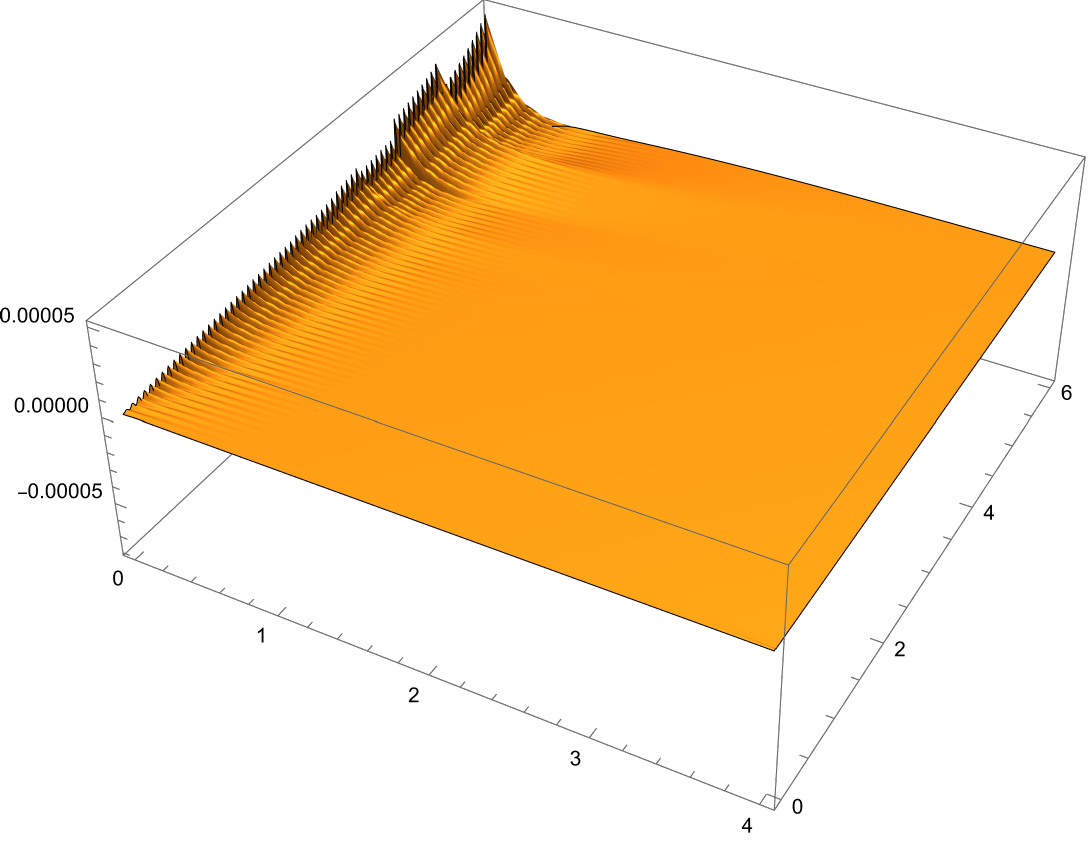}
    \hspace*{1cm}  \includegraphics[width=0.35\linewidth]{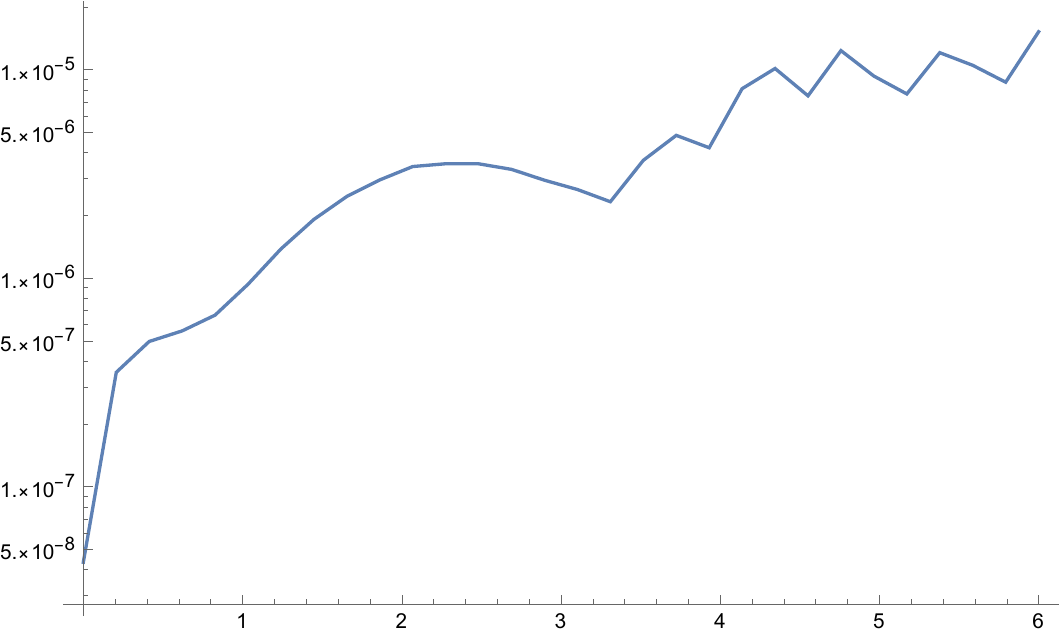}
    \caption{
    \textit{Top panel:} 
    Evaluation of the unified transform formulae~\eqref{heat-utm-t} and~\eqref{w-utm-t} for the solutions $v(x,t)$~(red) and $w(x,t)$ (blue) of the half-line problems~\eqref{vw-hl-ibvp-i} in the case of the boundary data~\eqref{ab-num}, which are obtained via the numerical solution of the integral equation~\eqref{a-int-eq} for $g(t) = \sin\left(2\pi t/\ell^2\right)$ supported for $t\in(0,T)$ with $T=3 \ell^2/8$. The surface colored in green corresponds to the sum $v(x,t)+w(x,t)$ and is virtually indistinguishable from the surface obtained by plotting the unified transform formula~\eqref{q-sol} for the solution $q(x,t)$ to the finite interval problem~\eqref{q-fi-ibvp-i} with boundary datum $g(t)$, thus verifying the decomposition \eqref{fi-hl-dec-i}. 
    \textit{Bottom panels:} The discrepancy $dc(x,t)=q(x,t)-\left[v(x,t)+w(x,t)\right]$ for $(x,t)\in (0,\ell)\times(0,T)$ with $\ell=4$ and $T = 6 < 4^2=\ell^2$, in line with the contraction condition~\eqref{T-L2} (bottom left), and the norm $\no{dc(t)}_{L_x^2(0,\ell)}$ for $t\in (0,T)$ in logarithmic scale (bottom right). Perfect agreement is observed between the numerical evaluation of the formulae for $q$ and $v+w$, consistent with the fact that the corresponding surfaces coincide in the top panel (green).
    }
    \label{fig:3D-ab-s}
\end{figure}

Furthermore, motivated by the global solvability of the integral equation \eqref{a-int-eq} established above, we compute $a(t)$ and $b(t)$ for values of $t$ which violate the restriction \eqref{T-L2}, namely $T\approx 42 > 16=\ell^2$,  where we observe bigger discrepancy between $a(t)$ and $g(t)$ --- in general, this also depends on the size of $a(t)$ itself. Similarly to the previous setup, in Figure \ref{fig:3D-ab-s-1} we plot the solutions $v(x,t)$, $w(x,t)$ to the two half-line problems \eqref{vw-hl-ibvp-i}, their sum $v(x,t)+ w(x,t)$, and the solution $q(x,t)$ to the finite interval problem \eqref{q-fi-ibvp-i}, with the surfaces corresponding to the last two quantities being virtually indistinguishable. The verification of the decomposition \eqref{fi-hl-dec-i} is illustrated by the bottom panels of Figure \ref{fig:3D-ab-s-1}, which display a discrepancy of  $\mathcal O(10^{-3})$.
\begin{figure}[ht!]
    \centering
    \includegraphics[scale=0.75]{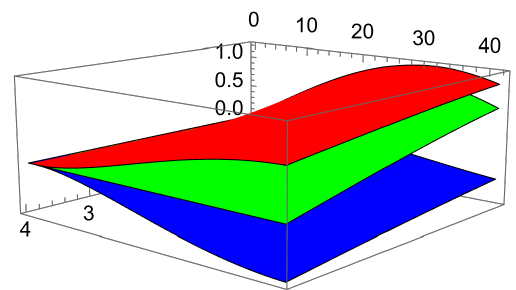} 
    \\
     \includegraphics[width=0.35\linewidth]{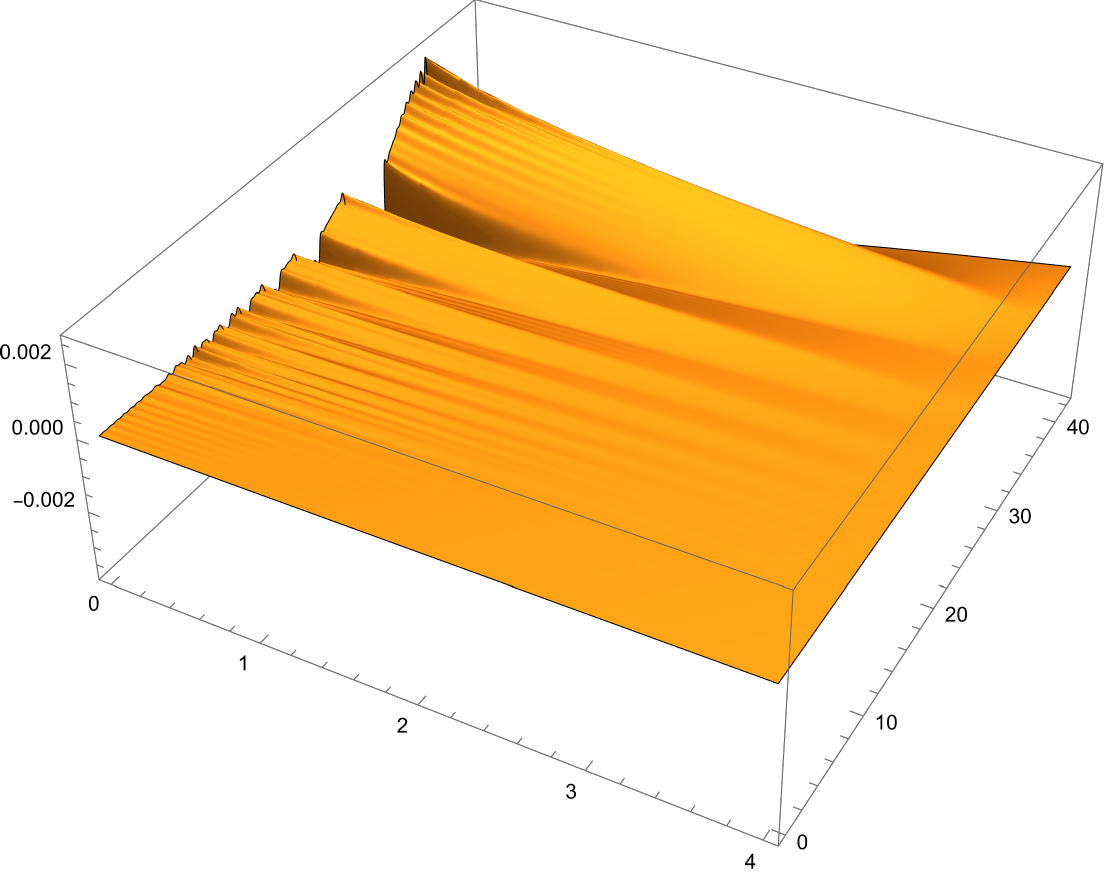}
    \hspace*{1cm}  \includegraphics[width=0.35\linewidth]{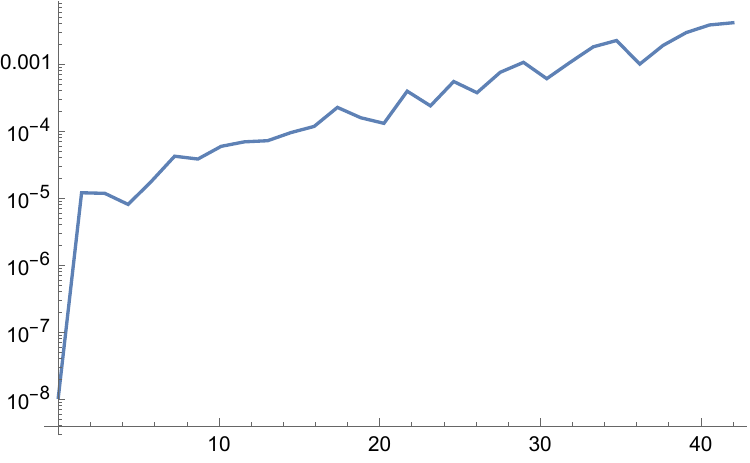}
    \caption{
   Same setup with the one of Figure \ref{fig:3D-ab-s} but for  $g(t) = \sin\left(2\pi t/(7\ell^2)\right)$, $t\in(0,T)$, with $T=21\ell^2/8$ so that $T = 42 > 4^2 = \ell^2$, thus violating the contraction condition~\eqref{T-L2}. Nevertheless, an excellent illustration of the decomposition \eqref{fi-hl-dec-i} is still observed, as the discrepancy $dc(x,t)=q(x,t)-v(x,t)-w(x,t)$ is of $\mathcal O(10^{-3})$.}
    \label{fig:3D-ab-s-1}
\end{figure}

\subsection{Evolution equations with time-dependent coefficients}
Let $d=d(t)$  be a smooth function on $[0,T]$ bounded below by a positive constant, namely $d(t) \geq c > 0$ for all $t\in [0, T]$. 
For the heat equation $u_t=d(t) u_{xx}$ with a time-variable diffusion coefficient $d(t)$,  based on the analysis of Section \ref{red-s} we discuss the reduction of the analysis of the finite interval problem to the one of the half-line problem.

In \cite{ko2025}, the following solution formulae were derived for the half-line problems  \eqref{vw-hl-ibvp-i} in the case of a variable diffusion coefficient:
\eee{\label{vsol-t}
\begin{aligned}
&v(x,t) = \frac{1}{i\pi} \int_{k\in \partial \mathcal D} e^{ikx - k^2 D(t)} k \left( \int_{z=0}^t d(z) e^{k^2 D(z)} a(z) dz \right) dk,
\\
&w(x,t) = \frac{1}{i\pi} \int_{k \in \partial \mathcal D} e^{ik (\ell-x) - k^2 D(t)} k \left( \int_{z=0}^t d(z) e^{k^2 D(z)} b(z) dz \right) dk,
\end{aligned}
\quad
D(t):= \int_0^t d(z) dz,
}
where $D(t)$ is an increasing function due to the positivity of $d(t)$ (this feature is essential to the derivation of formulae \eqref{vsol-t}). 

Via the approach of Section \ref{red-s}, we shall now derive the analogue of the integral equation \eqref{a-int-eq}. Consider the finite interval problem \eqref{q-fi-ibvp-i} for $q(x,t)$ but now in the case of a variable diffusion coefficient $d(t)$. Then, analogously to \eqref{ab-cond}, the decomposition \eqref{fi-hl-dec-i} requires that $b(t) = -v(\ell,t)$ where $v(x, t)$ satisfies the first of the half-line problems~\eqref{vw-hl-ibvp-i}, once again after adjusting the diffusion coefficient to $d(t)$. The solution to this half-line problem is given by the first of the formulae in \eqref{vsol-t}. Therefore, 
\eee{
b(t) =  -\frac{1}{i\pi} \int_{k\in\partial \mathcal D} e^{ik\ell - k^2 D(t)} k \left( \int_{z=0}^t d(z) e^{k^2 D(z)} a(z) dz \right) dk.
}
In turn, by the first of the conditions \eqref{ab-cond}, the second of the formulae \eqref{vsol-t} and the definition \eqref{dp-int} of $\Lambda_\ell(\sigma)$, we obtain the integral equation
\begin{equation}\label{a-int-eq-t}
    a(t) = g(t) + \frac{\ell^2}{4\pi} \int_{z=0}^t d(z) \, \Lambda_\ell\big(D(t)-D(z)\big)  \int_{r=0}^t d(r) \, \Lambda_\ell\big(D(z)-D(r)\big) \, a(r) \, dr \, dz,
\end{equation}
which provides the analogue of the integral equation \eqref{a-int-eq} that was solved in Section \ref{red-s}.

To this end, it is important to recall that for all $\sigma \in \mathbb R$ the integral $\Lambda_\ell(\sigma)$ admits the \textit{uniform} bound \eqref{L-bound}. Thus, we can prove the existence of solution to \eqref{a-int-eq-t} via a contraction mapping argument exactly as in the case of a unit diffusion coefficient. In particular, considering the map $a(t)\mapsto \Phi_g[a](t)$ with $\Phi_g[a]$ given by the right side of~\eqref{a-int-eq-t}, we may proceed as in Section \ref{cml2-ss} to derive the analogue of inequality \eqref{Phia-l2-est} in the form
\eee{
\no{\Phi_g[a]}_{L^2(0, T)}
\leq
\no{g}_{L^2(0, T)}
+
\frac{18 \sqrt 3 \, T^2}{\pi e^3 \ell^4}  \no{d}_{L^\infty(0, T)}^2 \no{a}_{L^2(0, T)}.
\label{Phia-l2-est-2-t}
}
Similarly, for any $a_1, a_2 \in L^2(0, T)$, we also have the following analogue to the contraction inequality \eqref{Phia-l2-contr}:
\eee{\label{Phia-l2-contr-t}
\no{\Phi_g[a_1] - \Phi_g[a_2]}_{L^2(0, T)}  
\leq
\frac{18 \sqrt 3 \, T^2}{\pi e^3 \ell^4} \no{d}_{L^\infty(0, T)}^2 \no{a_1-a_2}_{L^2(0, T)}.
}
The two inequalities \eqref{Phia-l2-est-2-t} and \eqref{Phia-l2-contr-t} can be used in the same way as their analogues in Section \ref{cml2-ss} to yield an $L^2(0, T)$ solution of the integral equation \eqref{a-int-eq-t} for a sufficiently small $T>0$ such that
\eee{
\frac{18 \sqrt 3 \, T^2}{\pi e^3 \ell^4} \no{d}_{L^\infty(0, T)}^2 \leq \frac 12.
}

\bibliographystyle{myamsalpha}
\bibliography{references_DM}

\end{document}